\documentclass[11pt,a4paper,reqno]{amsart}
\usepackage[english]{babel}
\usepackage[T1]{fontenc}
\usepackage{verbatim}
\usepackage{palatino}
\usepackage{amsmath}
\usepackage{mathabx}
\usepackage{amssymb}
\usepackage{amsthm}
\usepackage{amsfonts}
\usepackage{graphicx}
\usepackage{esint}
\usepackage{xcolor}
\usepackage{mathtools}
\usepackage{overpic}
\usepackage{enumerate}
\DeclarePairedDelimiter\ceil{\lceil}{\rceil}

\usepackage[colorlinks = true, citecolor = black]{hyperref}
\author{Tuomas Orponen and Pablo Shmerkin}
\title[Radial projections]{On exceptional sets of radial projections}
\address{Department of Mathematics and Statistics\\ University of Jyv\"askyl\"a,
P.O. Box 35 (MaD)\\
FI-40014 University of Jyv\"askyl\"a\\
Finland} \email{tuomas.t.orponen@jyu.fi}
\address{Department of Mathematics\\The University of British Columbia,
1984 Mathematics Road\\
Vancouver, BC, V6T 1Z2\\
Canada } \email{pshmerkin@math.ubc.ca}
\date{\today}
\thanks{T.O. is
 supported by the Academy of Finland via the projects
\emph{Quantitative rectifiability in Euclidean and non-Euclidean
spaces} and \emph{Incidences on Fractals}, grant Nos. 309365,
314172, 321896. }
\thanks{P.S. is supported by an NSERC Discovery Grant}

\newcommand{\R}{\mathbb{R}}

\newcommand{\N}{\mathbb{N}}

\newcommand{\spt}{\operatorname{spt}}
\newcommand{\Hd}{\dim_{\mathrm{H}}}

\newcommand{\diam}{\operatorname{diam}}

\newcommand{\dist}{\operatorname{dist}}

\def\Barint_#1{\mathchoice
          {\mathop{\vrule width 6pt height 3 pt depth -2.5pt
                  \kern -8pt \intop}\nolimits_{#1}}%
          {\mathop{\vrule width 5pt height 3 pt depth -2.6pt
                  \kern -6pt \intop}\nolimits_{#1}}%
          {\mathop{\vrule width 5pt height 3 pt depth -2.6pt
                  \kern -6pt \intop}\nolimits_{#1}}%
          {\mathop{\vrule width 5pt height 3 pt depth -2.6pt
                  \kern -6pt \intop}\nolimits_{#1}}}

\numberwithin{equation}{section}

\theoremstyle{plain}
\newtheorem{thm}[equation]{Theorem}
\newtheorem*{"thm"}{"Theorem"}

\newtheorem{lemma}[equation]{Lemma}

\newtheorem{cor}[equation]{Corollary}
\newtheorem{proposition}[equation]{Proposition}
\newtheorem{question}{Question}
\newtheorem{claim}[equation]{Claim}

\theoremstyle{definition}

\newtheorem{definition}[equation]{Definition}

\theoremstyle{remark}
\newtheorem{remark}[equation]{Remark}

\newcommand{\nref}[1]{(\hyperref[#1]{#1})}

\DeclareMathSymbol{\intop}  {\mathop}{mathx}{"B3}

\begin{document}

\begin{abstract} We prove two new exceptional set estimates for radial projections in the plane. If $K \subset \mathbb{R}^{2}$ is a Borel set with $\dim_{\mathrm{H}} K > 1$, then
\begin{displaymath} \dim_{\mathrm{H}} \{x \in \mathbb{R}^{2} \, \setminus \, K : \dim_{\mathrm{H}} \pi_{x}(K) \leq \sigma\} \leq \max\{1 + \sigma - \dim_{\mathrm{H}} K,0\}, \qquad \sigma \in [0,1). \end{displaymath}
If $K \subset \R^{2}$ is a Borel set with $\Hd K \leq 1$, then
\begin{displaymath} \dim_{\mathrm{H}} \{x \in \mathbb{R}^{2} \, \setminus \, K : \dim_{\mathrm{H}} \pi_{x}(K) < \dim_{\mathrm{H}} K\} \leq 1. \end{displaymath}
The finite field counterparts of both results above were recently proven by Lund, Thang, and Huong Thu. Our results resolve the planar cases of conjectures of Lund-Thang-Huong Thu, and Liu.
\end{abstract}

\maketitle

\tableofcontents

\section{Introduction}

It is a popular topic in fractal geometry to study the distortion of dimension under various projection maps. Perhaps the most basic family of projections consists of the orthogonal projections $x \mapsto \pi_{e}(x) := x \cdot e$ for $e \in S^{1}$. Another vigorously investigated family consists of the \emph{radial projections}. For $x \in \R^{2}$, the \emph{radial projection to $x$} is the map $\pi_{x} \colon \R^{2} \, \setminus \, \{x\} \to S^{1}$ defined by
\begin{displaymath} \pi_{x}(y) := \frac{y - x}{|y - x|}, \qquad y \in \R^{2} \, \setminus \, \{x\}. \end{displaymath}
Orthogonal and radial projections have higher-dimensional counterparts, but the discussion in this paper is mostly restricted to the plane. Regardless of the specific projection family under consideration, one often asks the following question: given a Borel set $K \subset \R^{2}$, how many elements in the projection family can lower the (Hausdorff) dimension of $K$ significantly? For orthogonal projections, we have the following three classic estimates for a Borel set $K \subset \R^{2}$ with $\Hd K = t$:

\begin{align}
\label{kaufman} \Hd \{e \in S^{1} : \Hd \pi_{e}(K) \leq \sigma\} \leq \sigma, &\qquad 0 \leq \sigma < \min\{t,1\},\\
\label{falconer} \Hd \{e \in S^{1} : \mathcal{H}^{1}(\pi_{e}(K)) = 0\} \leq 2 - t, &\qquad t > 1,
\end{align}
and
\begin{equation}
\label{peresSchlag} \Hd \{e \in S^{1} : \Hd \pi_{e}(K) \leq \sigma\} \leq \max\{1 + \sigma - t,0\}, \quad 0 \leq \sigma < \min\{t,1\}.
\end{equation}

Estimate \eqref{kaufman} is due to Kaufman \cite{Ka} from 1968, and it can be established with combinatorial or potential theoretic methods. The esimate \eqref{falconer} is due to Falconer \cite{MR673510} from 1982, and is based on Fourier analysis. A purely combinatorial proof is missing even today, as far as we know. The estimate \eqref{peresSchlag} is due to Peres and Schlag \cite{MR1749437} from 2000. It is also based on Fourier analysis, and can be viewed as a refinement of \eqref{falconer}. All the proofs of \eqref{kaufman}-\eqref{peresSchlag} can be conveniently read from Mattila's book \cite[Chapter 5]{MR3617376}.

Results on orthogonal projections always have corollaries to radial projections relative to points on a fixed line. Namely, if $\ell \subset \R^{2}$ is a line, then there exists a projective transformation $F_{\ell}$ with the property $\pi_{x}(K) = \pi_{e(x)}(F_{\ell}(K))$ for all $x \in \ell$ and $K \subset \R^{2} \, \setminus \, \ell$, where the map $x \mapsto e(x) \in S^{1}$ is locally bilipschitz. For more details, see \cite[Section 4]{MR3503710}. In particular, all the estimates \eqref{kaufman}-\eqref{peresSchlag} hold as stated if "$e \in S^{1}$" is replaced by "$x \in \ell$". Concretely, for example, Falconer's estimate \eqref{falconer} yields
\begin{equation}\label{radialFalconer} \Hd \{x \in \ell \, \setminus \, K : \mathcal{H}^{1}(\pi_{x}(K)) = 0\} \leq 2 - \Hd K, \qquad \Hd K > 1. \end{equation}
It is well-known (see \cite[p.114]{MR673510}) that Falconer's estimate \eqref{falconer} is sharp, and so is \eqref{radialFalconer}. Indeed, any sharpness example for \eqref{falconer} can be transformed into a sharpness example for \eqref{radialFalconer} using the map $F_{\ell}$. Given these facts, one might expect the following "global" estimate for radial projections to be optimal:
\begin{displaymath}
\Hd \{x \in \R^{2} \, \setminus \, K : \mathcal{H}^{1}(\pi_{x}(K)) = 0\} \leq 3 - \Hd K, \qquad \Hd K > 1.
\end{displaymath}
This estimate heuristically follows from \eqref{radialFalconer} by foliating $\R^{2}$ with parallel lines. Curiously, the intuition is false: in fact, the sharp "global" estimate for radial projections rather matches the numerology of \eqref{falconer} or \eqref{radialFalconer}:
\begin{equation}\label{orponen} \Hd \{x \in \R^{2} \, \setminus \, K : \mathcal{H}^{1}(\pi_{x}(K)) = 0\} \leq 2 - \Hd K, \qquad \Hd K > 1. \end{equation}
The estimate \eqref{orponen} was established in \cite{MR3778538}, but is based on the ideas in \cite{MR3503710}.

One might then guess that also the bounds \eqref{kaufman} and \eqref{peresSchlag} hold as stated for radial projections, just replacing "$e \in S^{1}$" by "$x \in \R^{2} \, \setminus \, K$". This is not the case: if $K \subset \R^{2}$ is compact and lies on a line $\ell \subset \R^{2}$, then the set $\Hd \{x \in \R^{2} \, \setminus \, K : \Hd \pi_{x}(K) = 0\}$ contains $\ell \, \setminus \, K$, and in particular has dimension $1$. This problem shows that neither \eqref{kaufman} nor \eqref{peresSchlag} can hold "globally" for radial projections if $\Hd K \leq 1$.

If $\Hd K > 1$, the counterexample $K \subset \ell$ is no longer possible. In this regime, the Peres-Schlag bound \eqref{peresSchlag} is always stronger than Kaufman's estimate \eqref{kaufman}. The first result of this paper shows that the analogue of \eqref{peresSchlag} holds "globally" for radial projections:

\begin{thm}\label{main} Let $K \subset \R^{2}$ be a Borel set with $\Hd K > 1$. Then,
\begin{displaymath} \Hd \{x \in \R^{2} \, \setminus \, K : \Hd \pi_{x}(K) \leq \sigma\} \leq \max\{1 + \sigma - \Hd K,0\}, \qquad 0 \leq \sigma < 1. \end{displaymath}
\end{thm}

The second result deals with sets of dimension at most one:
\begin{thm}\label{main3} Let $K \subset \R^{2}$ be a Borel set with $\Hd K \leq 1$. Then,
\begin{displaymath} \Hd \{x \in \R^{2} \, \setminus \, K : \Hd \pi_{x}(K) < \Hd K\} \leq 1. \end{displaymath}
\end{thm}
The finite field counterparts of Theorems \ref{main}-\ref{main3} (in all dimensions) were recently established by Lund, Thang, and Huong Thu \cite{2022arXiv220507431L}. Theorem \ref{main} resolves the planar case of \cite[Conjecture 1.2]{2022arXiv220507431L}. Theorem \ref{main3} resolves the planar case of \cite[Conjecture 1.2]{MR4269398}, and is sharp in cases where $K \subset \ell$ for some fixed line $\ell$, or more generally if $\Hd (\ell \, \setminus \, K) < \Hd K$. In contrast, it seems unlikely that Theorem \ref{main} is sharp for any fixed value $\sigma < 1$. However, it recovers the sharp estimate $\Hd \{x \in \R^{2} \, \setminus \, K : \Hd \pi_{x}(K) < 1\} \leq 2 - \Hd K$ as $\sigma \nearrow 1$.

The proof of Theorem \ref{main} is based on a reduction to a recent incidence estimate of Fu and Ren \cite{2021arXiv211105093F}, combined with some elementary estimates on \emph{Furstenberg sets} (these will be discussed later). The proof of Fu and Ren, further, involves a Fourier-analytic component, due to Guth, Solomon, and Wang \cite{GSW19}. So, while this paper contains no Fourier transforms, they play a role in the proof of Theorem \ref{main}.

Perhaps surprisingly, Theorem \ref{main3} is a corollary of Theorem \ref{main}. Its proof is based on a "swapping trick" introduced by Bochen Liu \cite{MR4269398}. Informally speaking, Liu's method is the following (see \cite[p. 9]{MR4269398}). Let $s,\sigma,t \geq 0$, and assume that $\Hd \pi_{x}(E) \geq \sigma$ for all Borel set $E \subset \R^{2}$ with $\Hd E = t$, outside an $s$-dimensional exceptional set. Then, $\Hd \pi_{x}(K) \geq 1 + \sigma - t$ for all Borel sets $K \subset \R^{2}$ with $\Hd K = s$, outside a $t$-dimensional exceptional set. We emphasise that this is only a heuristic principle!

Liu \cite[Theorem 1.1]{MR4269398} originally used his idea to prove a weaker version of Theorem \ref{main3}, where the number "$1$" on the right hand side is replaced by "$2 - \Hd K$". This result follows from \eqref{orponen} combined with Liu's method applied with $\sigma = 1$, $s = \Hd K$, and $t = 2 - s$. With Theorem \ref{main} in hand, we may instead apply the method with $\sigma = s = \Hd K$ and $t > 1$ arbitrarily close to "$1$". This explanation is grossly oversimplified: Liu \cite{MR4269398} used the proof of \eqref{orponen}, and not just the statement. We will need a $\delta$-discretised version of Theorem \ref{main}, recorded as Proposition \ref{prop4}.

\subsection{Sharper results for sets avoiding lines?} As we discussed earlier, Theorem \ref{main3} is sharp without further assumptions: if $K \subset \R^{2}$ is a compact, and lies on a line, then $\Hd \{x \in \R^{2} \, \setminus \, K : \Hd \pi_{x}(K) = 0\} = 1$. One may, then, ask what happens if the possibility of $K$ lying on a line is excluded. We are not sure what is the optimal formulation of the problem, but here is one possibility for concreteness:
\begin{question} Let $K \subset \R^{2}$ be a Borel set with $\Hd K \leq 1$, and assume that $\Hd (K \, \setminus \, \ell) = \Hd K$ for all lines $\ell \subset \R^{2}$. Is $\Hd \{x \in \R^{2} \, \setminus \, K : \Hd \pi_{x}(K) < \Hd K\} \leq \Hd K$?
\end{question}

The strongest current partial results in this direction are due to the second author and Wang \cite[Section 1.2]{2021arXiv211209044S}. For example, \cite[Theorem 1.7]{2021arXiv211209044S} states that if $K \subset \R^{2}$ is a Borel set with $\Hd K = 1$, and $K$ does not lie on a line, then there exists $x \in K$ such that $\Hd \pi_{x}(K \, \setminus \, \{x\}) \geq 0.618...$. An earlier result of the second author \cite[Theorem 6.13]{Shmerkin20} has similar flavour: if $K,E \subset \R^{2}$ are Borel sets of positive dimension, and $E$ is not contained on a line, then there exists $x \in E$ such that $\Hd \pi_{x}(K \, \setminus \, \{x\}) \geq \tfrac{1}{2} \Hd K + \eta$. Here $\eta$ only depends on $\Hd K$ and $\Hd E$. The same result with $\eta = 0$ was proven earlier by the first author in \cite{MR3892404}.

\subsection{Paper outline} Theorem \ref{main} is proved in Section \ref{deltaReduction} (where it is reduced to a $\delta$-discretised problem), and \ref{s:discrete} (where the $\delta$-discretised problem is solved by applying the incidence theorem of Fu and Ren \cite{2021arXiv211105093F}). Theorem \ref{main3} is proved in Section \ref{s:secondMain}.

\subsection{Notation} The notation $B(x,r)$ stands for a closed ball of radius $r > 0$ and centre $x \in X$, in a metric space $(X,d)$. If $A \subset X$ is a bounded set, and $r > 0$, we write $|A|_{r}$ for
the $r$-covering number of $A$, that is, the smallest number of
closed balls of radius $r$ required to cover $A$. Cardinality is
denoted $|A|$, and Lebesgue measure $\mathrm{Leb}(A)$. The closed
$r$-neighbourhood of $A$ is denoted $A(r)$. The (standard) dyadic
sub-cubes of $[0,1)^{d}$ of side-length $\delta \in 2^{-\N}$ will
be denoted $\mathcal{D}_{\delta}$. If $X,Y$ are positive numbers, then $X\lesssim Y$ means that $X\le CY$ for some constant $C$, while $X\gtrsim Y$, $X\sim Y$ stand for $Y\lesssim X$, $X\lesssim Y\lesssim X$. If the implicit constant $C$ depends on a parameter this will be mentioned explicitly or denoted by a subscript.

\subsection*{Acknowledgements} Some ideas for this paper were conceived while the authors were visiting the Hausdorff Research Institute for Mathematics, Bonn, during the trimester \emph{Interactions between Geometric measure theory, Singular integrals, and PDE.} We would like to thank the institute and its staff for their generous hospitality. We also thank Katrin F\"assler, Jiayin Liu, and Josh Zahl for useful discussions.

\section{Reduction to a $\delta$-discretised problem}\label{deltaReduction}

In this section, we reduce the proof of Theorem \ref{main} to a $\delta$-discretised statement. The main work is contained in Proposition \ref{prop2} below.

\begin{definition}[$(\delta,s,C)$-sets]\label{deltaSSet} Let $(X,d)$ be a metric space, let $P \subset X$ be a set, and let $\delta,C > 0$, and $s \geq 0$. We say that $P$ is a \emph{$(\delta,s,C)$-set} if
\begin{displaymath} |P \cap B(x,r)|_{\delta} \leq C r^{s} \cdot |P|_{\delta}, \qquad x \in X, \, r \geq \delta. \end{displaymath}
\end{definition}

The definition of a $(\delta,s,C)$-set, above, is slightly different from a more commonly used notion in the area, introduced by Katz and Tao \cite{MR1856956}. The Katz-Tao notion will also be used later in the paper, in Section \ref{s:discrete}. However, Definition \ref{deltaSSet} is not new either, and these variants of $(\delta,s,C)$-sets were for example employed in \cite{2021arXiv210603338O}. It is worth noting that a $(\delta,s,C)$-set is not required to be $\delta$-separated to begin with; however, it is easy to check that every $(\delta,s,C)$-set contains a $\delta$-separated $(\delta,s,C')$-set with $C' \sim C$. Another remark, which will be employed numerous times without explicit mention, is that if $P \subset \R^{d}$ is a $(\delta,s,C)$-set, and $P' \subset P$ satisfies $|P'|_{\delta} \geq c|P|_{\delta}$, then $P'$ is a $(\delta,s,C/c)$-set.

\begin{definition}[$(\delta,s,C)$-sets of lines and tubes] Let $\mathcal{A}(2,1)$ be the metric space of all (affine) lines in $\R^{2}$ equipped the metric
\begin{displaymath} d_{\mathcal{A}(2,1)}(\ell_{1},\ell_{2}) := \|\pi_{L_{1}} - \pi_{L_{2}}\| + |a_{1} - a_{2}|. \end{displaymath}
Here $\pi_{L_{j}} \colon \R^{2} \to L_{j}$ is the orthogonal projection to the subspace $L_{j}$ parallel to $\ell_{j}$, and $\{a_{j}\} = L_{j}^{\perp} \cap \ell_{j}$. A set $\mathcal{L} \subset \mathcal{A}(2,1)$ is called a $(\delta,s,C)$-set if $\mathcal{L}$ is a $(\delta,s,C)$-set in the metric space $(\mathcal{A}(2,1),d_{\mathcal{A}(2,1)})$, in the sense of Definition \ref{deltaSSet}.

A \emph{$\delta$-tube} is a (Euclidean) $\delta$-neighbourhood of some line $\ell \in \mathcal{A}(2,1)$. A family of $\delta$-tubes $\mathcal{T} = \{\ell(\delta) : \ell \in \mathcal{L}\}$ is called a $(\delta,s,C)$-set if the line family $\mathcal{L}$ is a $(\delta,s,C)$-set.
\end{definition}

Proposition \ref{prop2} below will discuss $(\delta,s,C)$-sets of $\delta$-tubes $\mathcal{T}_{x}$, $x \in B(1)$, with the special property that $x \in T$ for all $T \in \mathcal{T}_{x}$. In this case, it is easy to check that the $(\delta,s,C)$-set property of $\mathcal{T}_{x}$ is equivalent to the statement that the directions of the tubes (as a subset of $S^{1}$) form a $(\delta,s,C')$-set for some $C' \sim C$.

We then state the key technical proposition of the section:

\begin{proposition}\label{prop2} Let $t \in [0,2]$ and $\sigma \in (0,1)$. Let $K \subset B(1) \subset \R^{2}$ be a compact set with $\Hd K > t$, write
\begin{displaymath} E := \{x \in \R^{2} \, \setminus \, K : \Hd \pi_{x}(K) < \sigma\}, \end{displaymath}
and assume that $\Hd E > s \geq 0$. Then, for every $\delta_{0} \in (0,1)$ and $\epsilon \in (0,1 - \sigma)$, the following objects exist for some $\Delta \in (0,\delta_{0}]$, and for some $\Sigma \leq \sigma + \epsilon$:
\begin{enumerate}[(i)]
\item \label{i} A non-empty $\Delta$-separated $(\Delta,t,\Delta^{-\epsilon})$-set $P_{K} \subset K$,
\item \label{ii} A non-empty $\Delta$-separated $(\Delta,s,\Delta^{-\epsilon})$-set $P_{E} \subset E$,
\item \label{iii} For every $x \in P_{E}$ a non-empty $(\Delta,\Sigma,\Delta^{-\epsilon})$-set of $\Delta$-tubes $\mathcal{T}_{x}$ with the properties that $x \in T$ for all $T \in \mathcal{T}_{x}$, and
\begin{displaymath} |T \cap P_{K}| \geq \Delta^{\Sigma + \epsilon}|P_{K}|, \qquad T \in \mathcal{T}_{x}. \end{displaymath}
\end{enumerate}
\end{proposition}

\begin{remark} It may be informative to note that $1 + \sigma - t > 0$ under the hypothesis of Proposition \ref{prop2}. Indeed, since $\Hd E > 0$ by assumption, there exists at least one point $x \in E$. This forces $t < \Hd K \leq 1 + \sigma$, and consequently $1 + \sigma - t > 0$. \end{remark}


Before proving Proposition \ref{prop2}, we state a $\delta$-discretised version of Theorem \ref{main}. Then, we conclude the proof of Theorem \ref{main}, assuming this statement, and Proposition \ref{prop2}.

\begin{thm}\label{main2}
For every $t \in (1,2]$, $\sigma \in [0,1)$, and $\eta > 0$,
there exist $\epsilon = \epsilon(\sigma,t,\eta) > 0$ and $\delta_{0} = \delta_{0}(\sigma,t,\eta) > 0$
such that the following holds for all $\delta \in (0,\delta_{0}]$.

Let $s \in [0,2]$. Let $P_{K} \subset B(1) \subset \R^{2}$ be a $\delta$-separated $(\delta,t,\delta^{-\epsilon})$-set, and let $P_{E} \subset B(1)
\subset \R^{2}$ be a $\delta$-separated $(\delta,s,\delta^{-\epsilon})$-set. Assume
that for every $x \in P_{E}$, there exists a
$(\delta,\sigma,\delta^{-\epsilon})$-set of tubes
$\mathcal{T}_{x}$ with the properties $x \in T$ for all $T \in
\mathcal{T}_{x}$, and
\begin{displaymath} |T \cap P_{K}| \geq \delta^{\sigma + \epsilon}|P_{K}|, \qquad T \in \mathcal{T}_{x}. \end{displaymath}
Then $s \leq 1 + \sigma - t + \eta$.
\end{thm}

The proof of Theorem \ref{main2} is contained in Section \ref{s:discrete}.

\begin{proof}[Proof of Theorem \ref{main} assuming Theorem \ref{main2} and Proposition \ref{prop2}] We start with a counter assumption. There exists a compact set $K \subset B(1)$ with $\Hd K > t > 1$, a number $\sigma \in [0,1)$, and another number $s > \max\{1 + \sigma - t,0\}$ such that the set
\begin{displaymath} E = \{x \in B(1) \, \setminus \, K : \Hd \pi_{x}(K) < \sigma\} \end{displaymath}
satisfies $\Hd E > s$. In particular, $s > 1 + \sigma - t + 2\eta$ for some $\eta > 0$.

With this information in hand, and since also $s \geq 0$, Proposition \ref{prop2}
 yields (for any given $\delta_{0} \in (0,1)$ and $\epsilon \in (0,1 - \sigma)$)
the objects $P_{K} \subset K \subset B(1)$ and $P_{E} \subset E
\subset B(1)$, satisfying the properties \eqref{i}--\eqref{iii} of Proposition
\ref{prop2} at some scale $\Delta \in (0,\delta_{0}]$, and for
some $\Sigma \leq \sigma + \epsilon$. In particular, if we choose
$\epsilon \in (0,\eta]$, we have
\begin{displaymath} s > 1 + \Sigma - t + \eta. \end{displaymath}
However, this violates the statement of Theorem \ref{main2} with parameters "$\Delta$" and "$\Sigma$" in place of "$\delta$" and "$\sigma$", assuming that $\epsilon,\delta_{0}$ were taken small enough to begin with, depending on $\sigma,t$, and $\eta$. This contradiction completes the proof of Theorem \ref{main}. \end{proof}

\subsection{Proof of Proposition \ref{prop2}} The next notion will be useful in proving Proposition \ref{prop2}.
\begin{definition}[Uniform measures]\label{d:regularity} Let $d \in \N$, and let $\Phi := (\varphi(1),\ldots,\varphi(m)) \in \R_{+}^{m}$. Let $\mu$ be a Radon measure on $[0,1)^{d}$, and let $\eta \in (0,1)$, $T \geq 1$. The measure $\mu$ is called $(\Phi,T,\eta)$-uniform if for every $k \in \{\ceil{\eta m},\ldots,m\}$, and every $Q \in \mathcal{D}_{kT}$ either
\begin{equation}\label{uniformity} \ell(Q)^{\varphi(k) + \eta} \leq \mu(Q) \leq \ell(Q)^{\varphi(k)} \quad \text{or} \quad \mu(Q) = 0. \end{equation}
\end{definition}

Definition \ref{d:regularity} is somewhat reminiscent of the definition of "$\sigma$-regular measures" introduced by Keleti and Shmerkin in \cite[Definition 3.2]{KeletiShmerkin19}. The main difference is that in \cite[Definition 3.2]{KeletiShmerkin19}, the sequence $\Phi$ encodes the ratios $\mu(Q)/\mu(\hat{Q})$ for $Q \in \mathcal{D}_{kT}$, whereas our definition makes no reference to this ratio.

We also note that the $(\Phi,T,\eta)$-regularity of $\mu$ requires nothing about the values of $\mu$ on cubes whose side-length is greater than $2^{-\ceil{\eta m}T}$. In particular the values $\varphi(k)$ with $k < \ceil{\eta m}$ are not actually used in Definition \ref{d:regularity}.

The next lemma allows us to find $(\Phi,T,\eta)$-uniform "sub-measures" of arbitrary Radon measures. Analogous statements for other notions of "regular measures" and "uniform sets" are common in the literature, see for example \cite[Lemma 3.4]{KeletiShmerkin19} or \cite[Lemma 7.3]{2021arXiv210603338O}.

\begin{lemma}\label{lemma2} For every $d \in \N$ and $\eta \in (0,1]$, there exists $T_{0} = T_{0}(d,\eta) \geq 1$ (see \eqref{form21}) such that the following holds for $T \geq T_{0}$. Let $\mu$ be a Radon measure with $\spt \mu \subset [0,1)^{d}$ and $\delta^{\eta} \leq \|\mu\| \leq 1$, where $\delta = 2^{-mT}$ for some $m \in \N$. Then, there exists a family of dyadic cubes $\mathcal{Q} \subset \mathcal{D}_{\delta}$ with the following properties:
\begin{itemize}
\item The restriction $\mu' := \mu|_{\cup \mathcal{Q}}$ has mass $\|\mu'\| \geq \delta^{2\eta}$,
\item $\mu'$ is $(\Phi,T,\eta)$-uniform with some $\Phi \subset [0,d + 2]^{m}$.
\end{itemize}
\end{lemma}

\begin{proof} The proof is based on $\sim m$ applications of the pigeonhole principle, although one needs to be somewhat careful with the constants. We set
\begin{equation}\label{form21} T_{0} := \ceil{\eta^{-1} \log \tfrac{10d}{\eta}}, \end{equation}
and fix $T \geq T_{0}$. We start the inductive construction of the measure $\mu'$ at scale $\delta = 2^{-mT}$. We set $\mu_{m + 1} := \mu$. Then, we assume that $k \in \{\ceil{\eta m},\ldots,m\}$, a measure $\mu_{k + 1}$ has already been constructed, and satisfies
\begin{equation}\label{form22} \|\mu_{k + 1}\| \geq \delta^{\eta}\left(\frac{\eta}{10d} \right)^{m - k} \geq 2\delta^{\eta}\left(\frac{\eta}{10d} \right)^{m}. \end{equation}
The measure $\mu_{k}$ will be obtained by restricting $\mu_{k +
1}$ to a certain union of cubes in $\mathcal{D}_{kT}$. We make the
following preliminary observation: the union of the "light" cubes
in $Q \in \mathcal{D}_{kT}$ satisfying $\mu(Q) < \ell(Q)^{d + 2}$
has total $\mu$ measure at most
\begin{displaymath} \ell(Q)^{-d} \cdot \ell(Q)^{d + 2} = \ell(Q)^{2} = 2^{-2kT} \leq 2^{-2\ceil{\eta m}T} \leq \delta^{2\eta}. \end{displaymath}
On the other hand, we claim that by the choice of $T = T(d,\eta)$ in \eqref{form21}, we have
\begin{equation}\label{form20} \delta^{2\eta} \leq \delta^{\eta}\left(\frac{\eta}{10d} \right)^{m} \stackrel{\eqref{form22}}{\leq} \tfrac{1}{2}\|\mu_{k + 1}\|. \end{equation}
Indeed, since $\delta = 2^{-mT}$, we have $m = \tfrac{1}{T}\log \tfrac{1}{\delta}$. Therefore, the first inequality in \eqref{form20} is equivalent to
\begin{displaymath} \eta \log \delta \leq \tfrac{1}{T} (\log \tfrac{1}{\delta}) (\log \tfrac{\eta}{10d}) \quad \Longleftrightarrow \quad T \geq \eta^{-1} \log \tfrac{10d}{\eta},\end{displaymath}
and the latter inequality is true by $T \geq T_{0}$ and \eqref{form21}. The conclusion of all this is that the "light" cubes cover at most $\tfrac{1}{2}$ of the total mass of $\mu_{k + 1}$. Consequently, the union of the "heavy" cubes
$Q \in \mathcal{D}_{kT}$ satisfying $\ell(Q)^{d + 2} \leq \mu(Q) \leq \ell(Q)^{0}$ has measure at least $\tfrac{1}{2}\|\mu_{k + 1}\|$. We also note that
\begin{displaymath} |(\eta \N) \cap [0,d + 2]| \leq 5d/\eta. \end{displaymath}
Therefore, by the pigeonhole principle, there exists a constant $\varphi(k) \in (\eta \N) \cap [0,d + 2]$, and a family $\mathcal{Q}_{k} \subset \mathcal{D}_{kT}$ satisfying
\begin{equation}\label{form23} \mu_{k + 1}
\left(\cup \mathcal{Q}_{k} \right) \geq \frac{\eta}{5d} \cdot
\tfrac{1}{2}\|\mu_{k + 1}\| \stackrel{\eqref{form22}}{\geq}
\frac{\eta}{10d}\cdot \delta^{\eta} \cdot \left(\frac{\eta}{10d}
\right)^{m - k} = \left(\frac{\eta}{10d} \right)^{m - k + 1}
\delta^{\eta}
\end{equation} and
\begin{equation}\label{form24} \ell(Q)^{\varphi(k) + \eta} \leq \mu_{k + 1}(Q) \leq \ell(Q)^{\varphi(k)}, \qquad Q \in \mathcal{Q}_{k}. \end{equation}
We now define $\mu_{k}$ as the restriction of $\mu_{k + 1}$ to the union $\cup \mathcal{Q}_{k}$. Then, the inequality \eqref{form23} shows that the inductive hypothesis \eqref{form22} remains valid with $\mu_{k}$ in place of $\mu_{k + 1}$.

After the construction above has been carried out for all $k \in \{\ceil{\eta m},\ldots,m\}$, we define $\mu' := \mu_{\ceil{\eta m}}$. The inequalities \eqref{form24} show that $\mu'$ is $(\Phi,T,\eta)$-uniform. To see this, note that $\mu_{k}$ is obtained by restricting $\mu_{k + 1}$ to a union of cubes in $\mathcal{D}_{kT}$, the values of $\mu_{k}$ on cubes in $\mathcal{D}_{(k + 1)T}$ either stay unchanged, or are assigned to $0$. This, and \eqref{form24}, implies that
\begin{displaymath} \ell(Q)^{\varphi(k) + \eta} \leq \mu'(Q) \leq \ell(Q)^{\varphi(k)} \quad \text{or} \quad \mu'(Q) = 0 \end{displaymath}
for all $Q \in \mathcal{D}_{kT}$ with $k \in \{\ceil{\eta m},\ldots,m\}$.

Finally, $\|\mu'\| \geq \delta^{2\eta}$. This is immediate from the estimate \eqref{form20}. This completes the proof of the lemma. \end{proof}

We may then prove Proposition \ref{prop2}:

\begin{proof}[Proof of Proposition \ref{prop2}] Since $K$ was assumed compact, it is easy to check that $E = \{x \in B(1) \, \setminus \, K : \Hd \pi_{x}(K) < \sigma\}$ is Borel. By assumption, $\Hd K > t$ and $\Hd E > s$. By Frostman's lemma \cite[Theorem 8.8]{zbMATH01249699}, there exist non-trivial compactly supported Borel measures $\mu$ and $\nu$ such that $\spt \mu \subset K$, $\spt \nu \subset E$, $\max\{\|\mu\|,\|\nu\|\} \leq 1$, and
\begin{displaymath} \mu(B(x,r)) \leq r^{t} \quad \text{and} \quad \nu(B(x,r)) \leq r^{s} \end{displaymath}
for all $x \in \R^{2}$ and $r > 0$. For the remainder of this argument, the constants in the "$\lesssim$" notation are allowed to depend on $\|\mu\|$ and $\|\nu\|$. The size of $\|\mu\|$ and $\|\nu\|$ certainly influences the best constant "$C$" for which $P_{K}$ and $P_{E}$ are $(\delta,u,C)$-sets. However, the argument below will let us find "$\Delta$" as small as we like, and for example we may choose $\Delta > 0$ so small that $\Delta^{-\epsilon} \geq \max\{O(\|\mu\|^{-1}),O(\|\nu\|^{-1})\}$. We will find the sets $P_{K},P_{E}$ inside $\spt \mu$ and $\spt \nu$, so we redefine $K := \spt \mu$ and $E := \spt \nu$ (the assumptions $\Hd K > t$ and $\Hd E > s$ were only needed to find the measures $\mu,\nu$). Since $K$ and $E$ are now compact and disjoint,
\begin{displaymath} D := \dist(K,E) > 0. \end{displaymath}

We then start the proof in earnest by fixing some parameters. First, let $\delta_{0},\epsilon > 0$ be the parameters given in the statement of Proposition \ref{prop2}. Next, choose $\eta = \eta(\epsilon) > 0$ much smaller that "$\epsilon$", more precisely
\begin{equation}\label{eta} 0 < \eta < \tfrac{1}{4} \cdot \epsilon^{3/\epsilon + 1}. \end{equation}
There will be several occasions in the proof where we need to assume that $\delta > 0$ is small enough in a manner depending on the constants $\delta_{0},D,\eta,\epsilon$. This dependence is simply referred to by writing "if $\delta > 0$ is small enough". One particularly important restriction of this kind is the following:
\begin{equation}\label{form25} \delta^{\eta} \leq \delta_{0}. \end{equation}
Next, let $T = T(1,\eta) > 0$ be the number given by Lemma \ref{lemma2} applied to the constant "$\eta$", and in the space $S^{1}$ (identified with $[0,1)$ if desired, although the proof would work on $S^{1}$ equally well). We will apply the lemma to measures of the form $\pi_{x}\mu$ supported on $S^{1}$. In this application, we will denote by $\mathcal{D}_{2^{-j}}$ the system of "dyadic arcs" on $S^{1}$.

To apply the lemma, fix $x \in E$. By the definition of $\Hd \pi_{x}(K) < \sigma$, and the pigeonhole principle, there exist dyadic scales $\delta < \delta_{0}$ of the form $\delta = 2^{-mT}$, satisfying \eqref{form25}, and associated collections $\mathcal{J}_{x} \subset \mathcal{D}_{\delta}$  of dyadic arcs of length $\delta$, such that
\begin{equation}\label{form8} |\mathcal{J}_{x}| \leq \delta^{-\sigma} \quad \text{and} \quad (\pi_{x}\mu)\left(\cup \mathcal{J}_{x} \right) \geq \delta^{\eta}. \end{equation}
The scale $\delta = \delta_{x} > 0$ depends on $x \in E$, but by another application of the pigeonhole principle, there exists a subset $E_{0} \subset E$ with $\nu(E_{0}) \gtrsim (\log (1/\delta))^{-2}$ such that $\delta_{x} = \delta$ is independent of $x \in E_{0}$. To simplify notation, we keep denoting $E_{0}$ by $E$, but when the $\nu$ measure of $E$ next plays a role (all the way down at \eqref{form53}), we will keep in mind the factor $(\log(1/\delta))^{-2}$; this will be completely harmless.

In the sequel, we denote $\mu_{x}$ the restriction of $\pi_{x}\mu$ to the union $\cup \mathcal{J}_{x}$, so $\delta^{\eta} \leq \|\mu_{x}\| \leq 1$. This implies that we may apply Lemma \ref{lemma2} to the measure $\mu_{x}$, for $x \in E$ fixed. The conclusion is the existence of a measure $\mu_{x}'$, which
\begin{enumerate}[(a)]
\item\label{a} is the restriction of $\mu_{x}$ to a sub-family $\mathcal{J}_{x}' \subset \mathcal{J}_{x}$,
\item\label{b} satisfies $\|\mu_{x}'\| \geq \delta^{2\eta}$, and
\item\label{c} is $(\Phi,T,\eta)$-uniform for some $\Phi \subset [0,3]^{m}$, i.e. has the property \eqref{uniformity} for some exponents $\varphi_{x}(k) \in [0,3]$, for $k \in \{\ceil{\eta m},\ldots,m\}$.
\end{enumerate}
We will abuse notation by writing $\varphi_x(\delta) :=
\varphi_x(m)$, and more generally $\varphi_x(2^{-kT}) :=
\varphi_x(k)$. We define the sets
\begin{equation}\label{form12} K_{x} := K \cap \pi_{x}^{-1}\left(\cup \mathcal{J}_{x}' \right), \qquad x \in E. \end{equation}
With this notation,
\begin{equation}\label{form18} \mu_{x}' = \pi_{x}(\mu|_{K_{x}}), \end{equation}
and consequently $\mu(K_{x}) = \|\mu_{x}'\| \geq \delta^{2\eta}$ according to points \eqref{a}-\eqref{b}. We next observe that $\varphi_{x}(\delta) \leq \sigma + 2\eta \leq \sigma + \epsilon$, because
\begin{displaymath} \delta^{2\eta} \leq \|\mu_{x}'\| \leq \delta^{\varphi_{x}(\delta)}|\mathcal{J}_{x}| \stackrel{\eqref{form8}}{\leq} \delta^{\varphi_{x}(\delta) - \sigma}. \end{displaymath}
Next, we claim the following: for every $x \in E$, there exists a scale
\begin{equation}\label{form9} \Delta_{x} \in 2^{-\N \cdot T} \cap [\delta,\delta^{\epsilon^{3/\epsilon}}] \stackrel{\eqref{eta}}{\subset} 2^{-\N \cdot T} \cap [\delta,\delta^{\eta}] \stackrel{\eqref{form25}}{\subset} [\delta,\delta_{0}] \end{equation}
with the properties $\varphi_x(\Delta_{x}) \leq \varphi_x(\delta)$
and
\begin{equation}\label{form10} \varphi_x(\rho) \geq \varphi_x(\Delta_{x}) - \epsilon, \qquad \rho \in 2^{-\N \cdot T} \cap [\Delta_{x},\Delta_{x}^{\epsilon}]. \end{equation}
We mention a subtle point straight away: for $\Delta_{x} \in 2^{-\N \cdot T} \cap [\delta,\delta^{\epsilon^{3/\epsilon}}]$, we have
\begin{equation}\label{form27} \Delta_{x}^{\epsilon} \leq \delta^{\epsilon^{3/\epsilon + 1}} \leq \delta^{\eta} \end{equation}
by the choice of $\eta$ at \eqref{eta}. Therefore the values
"$\varphi_x(\rho)$" for $\rho \in 2^{-\N \cdot T} \cap
[\Delta_{x},\Delta_{x}^{\epsilon}]$ appearing in \eqref{form10}
are always related to the values of $\mu_{x}'(Q)$ for $Q \in
\mathcal{D}_{\rho}$ via the formula \eqref{uniformity}. This will
be needed a little later.

The scale $\Delta_{x}$ is produced by the following algorithm. First, we check if $\Delta_{x} := \Delta_{0} := \delta$ would work, and satisfy \eqref{form10}. If not, then there exists $\Delta_{1} = 2^{-k_{1}T} \in [\Delta_{0},\Delta_{0}^{\epsilon}]$ with the property
\begin{displaymath} \varphi_x(\Delta_{1}) < \varphi_x(\Delta_{0}) - \epsilon. \end{displaymath}
At this point, we check if $\Delta_{x} := \Delta_{1}$ works. If we
are still out of luck, we find $\Delta_{2} = 2^{-k_{2}T} \in
[\Delta_{1},\Delta_{1}^{\epsilon}]$ violating \eqref{form10}.
Continuing in this manner, we construct an increasing sequence of
scales $\Delta_{0} \leq \Delta_{1} \leq \ldots$ with the
properties $\Delta_{j + 1} = 2^{-k_{j + 1}T} \in
[\Delta_{j},\Delta_{j}^{\epsilon}]$, and $\varphi_x(\Delta_{j + 1})
\leq \varphi_x(\Delta_{j}) - \epsilon$. Since on the other hand
$\varphi_x(\Delta_{0}) \leq 3$, and $\varphi_x(\rho) \geq 0$ for
all $\rho = 2^{-kT}$, the inequality $\varphi_x(\Delta_{j + 1})
\leq \varphi_x(\Delta_{j}) - \epsilon$ can only hold for $\leq
3/\epsilon$ consecutive indices "$j$". In other words, the
algorithm must terminate in $N \leq 3/\epsilon$ steps. At this
stage, the scale $\Delta_{x} = \delta_{N} = 2^{-k_{N}T}$ satisfies
\eqref{form9}-\eqref{form10}.

Now that $\Delta_{x}$ has been found and fixed, we set
\begin{displaymath} \Sigma_{x} := \varphi_x(\Delta_{x}). \end{displaymath}
The number
\begin{equation}\label{Sigma} \Sigma_{x} \leq \varphi_{x}(\delta) \leq \sigma + \epsilon \leq 1 \end{equation}
is a candidate for the number "$\Sigma$" appearing in the statement of the proposition. Let $\mathcal{I}_{x} \subset \mathcal{D}_{\Delta_{x}}$ be the collection of dyadic $\Delta_{x}$-arcs $I \subset S^{1}$ satisfying $\mu_{x}'(I) > 0$, or equivalently (by the $(\Phi,T,\eta)$-regularity of $\mu_{x}'$)
\begin{equation}\label{form13} \Delta_{x}^{\Sigma_{x} + \eta} \leq \mu_{x}'(I) \leq \Delta_{x}^{\Sigma_{x}}. \end{equation}
Here it is important to note that $\Delta_{x} \leq \delta^{\eta}$, so the $(\Phi,T,\eta)$-regularity of $\mu_{x}'$ still says something, namely \eqref{form13}, about arcs of length $\Delta_{x}$. We also set $S_{x} := \cup \mathcal{I}_{x}$. It follows immediately from $\|\mu_{x}'\| \leq \|\mu\| \leq 1$, and the lower bound in \eqref{form13} that $|S_{x}|_{\Delta_{x}} \sim |\mathcal{I}_{x}| \leq \Delta_{x}^{-\Sigma_{x} - \eta}$. We then make the following claim:
\begin{claim}\label{SClaim} If $\delta > 0$ is small enough, then $S_{x}$ is a $(\Delta_{x},\Sigma_{x},\Delta_{x}^{-2\epsilon})$-set.
\end{claim}
These facts are a precedent of the property \eqref{iii} in the statement of the proposition. First, we observe the following lower bound for the $\Delta_{x}$-covering number of $S_{x}$:
\begin{displaymath} \delta^{2\eta} \leq \|\mu_{x}'\| \stackrel{\eqref{form13}}{\leq} |\mathcal{I}_{x}| \cdot \Delta_{x}^{\Sigma_{x}} \quad \Longrightarrow \quad |S_{x}|_{\Delta_{x}} \sim |\mathcal{I}_{x}| \geq \delta^{2\eta} \cdot \Delta_{x}^{-\Sigma_{x}}. \end{displaymath}
It might appear problematic here that "$\delta$" is potentially much smaller than $\Delta_{x}$, but one should recall from \eqref{form9} that $\Delta_{x} \leq \delta^{\epsilon^{3/\epsilon}}$, and in \eqref{eta} we chose $\eta < \epsilon^{3/\epsilon + 1}/4$. Therefore,
\begin{equation}\label{form11}
\delta^{2\eta} \geq \Delta_{x}^{2\eta \cdot \epsilon^{-3/\epsilon}} \geq \Delta_{x}^{\epsilon/2} \quad \Longrightarrow \quad |S_{x}|_{\Delta_{x}} \gtrsim \Delta_{x}^{-\Sigma_{x} + \epsilon/2}.
\end{equation}
Now, to prove that $S_{x}$ is $(\Delta_{x},\Sigma_{x},\Delta_{x}^{-2\epsilon})$-set, fix $e \in S^{1}$ and $\Delta_{x} \leq \rho \leq 1$. If $\rho > \Delta_{x}^{\epsilon}$, we simply use $\Sigma_{x} \leq \sigma + \epsilon \leq 1$ (recall \eqref{Sigma}) to estimate
\begin{displaymath} |S_{x} \cap B(e,\rho)|_{\Delta_{x}} \leq |S_{x}|_{\Delta_{x}} \stackrel{\eqref{form13}}{\lesssim} \Delta_{x}^{-\Sigma_{x} - \eta} \leq \frac{\rho}{\Delta_{x}^{\epsilon}} \cdot \Delta_{x}^{-\Sigma_{x} - \eta} \stackrel{\eqref{form11}}{\lesssim} \Delta_{x}^{-3\epsilon/2 - \eta} \cdot \rho^{\Sigma_{x}} |S_{x}|_{\Delta_{x}}. \end{displaymath}
In particular, $|S_{x} \cap B(e,\rho)|_{\Delta_{x}} \leq \Delta_{x}^{-2\epsilon} \cdot \rho^{\Sigma_{x}}|S_{x}|_{\Delta_{x}}$, if $\delta > 0$ is small enough. On the other hand, if $\Delta_{x} \leq \rho \leq \Delta_{x}^{\epsilon}$, then $\rho \leq \delta^{\eta}$, as noted at \eqref{form27}. This means that the $(\Phi,T,\eta)$-regularity of $\mu_{x}'$ still says something useful at scale $\rho$. To make this precise, let $\bar{\rho} = 2^{-kT}$ satisfy $\Delta_{x} \leq \bar{\rho} \leq \rho \leq 2^{T}\bar{\rho}$. Then $\bar{\rho} \in 2^{-\N \cdot T} \cap [\Delta_{x},\Delta_{x}^{\epsilon}]$, so $\varphi_{x}(\bar{\rho}) \geq \Sigma_{x} - \epsilon$ by \eqref{form10}. Consequently
\begin{displaymath} \mu_{x}'(B(e,2\rho)) \lesssim_{T} \max_{Q \in \mathcal{D}_{\bar{\rho}}} \mu_{x}'(Q) \stackrel{\eqref{uniformity}}{\leq} \bar{\rho}^{\varphi(\bar{\rho})} \stackrel{\eqref{form10}}{\leq} \bar{\rho}^{-\epsilon}\rho^{\Sigma_{x}} \leq \Delta_{x}^{-\epsilon}\rho^{\Sigma_{x}}. \end{displaymath}
Since on the other hand $\mu_{x}'(I) \geq \Delta_{x}^{\Sigma_{x} + \eta}$ for all $I \in \mathcal{I}_{x}$, recall \eqref{form13}, we see that
\begin{displaymath} |S_{x} \cap B(e,\rho)|_{\Delta_{x}} \lesssim \frac{\mu_{x}'(B(e,2\rho))}{\Delta_{x}^{\Sigma_{x} + \eta}}
 \lesssim_{T} \Delta_{x}^{-\Sigma_{x} - \epsilon - \eta} \rho^{\Sigma_{x}} \stackrel{\eqref{form11}}{\lesssim}
  \Delta_{x}^{-3\epsilon/2 - \eta} |S_x|_{\Delta_{x}}. \end{displaymath}
Since $\eta < \epsilon/4$, we may finally estimate $C_{T}\Delta_{x}^{-\eta} \leq \Delta_{x}^{-\epsilon/2}$ for $\delta > 0$ small enough. This completes the proof of the $(\Delta_{x},\Sigma_{x},\Delta_{x}^{-2\epsilon})$-property of $S_{x}$.

To make the connection to property \eqref{iii} in the statement of Proposition \ref{prop2} clearer, we now define a family of $\Delta_{x}$-tubes $\mathcal{T}_{x}$, all containing $x$. Recall that $K \subset B(1)$ and $\dist(K,E) = D > 0$. This implies that each set of the form
\begin{displaymath} K \cap \pi_{x}^{-1}(I), \qquad I \in \mathcal{I}_{x}, \end{displaymath}
can be covered by a single tube $T = T_{I}$ of width $\sim_{D} \Delta_{x}$ containing the point $x$. We let $\mathcal{T}_{x}$ be the collection of tubes so obtained, for all $I \in \mathcal{I}_{x}$. We record an upper bound for $|\mathcal{T}_{x}|$. The intersections $\{K \cap T : T \in \mathcal{T}_{x}\}$ have bounded overlap, with overlap constant depending on "$D$". Moreover, each intersection $K \cap T$ has $\mu$ measure no smaller than $\mu_{x}'(I) \geq \Delta_{x}^{\Sigma_{x} + \eta}$ (for some $I \in \mathcal{I}_{x}$). It follows that
\begin{equation}\label{form36} |\mathcal{T}_{x}| \lesssim_{D} \Delta_{x}^{-\Sigma_{x} - \eta}. \end{equation}
By the bounded overlap property mentioned above, each intersection $K \cap T$ with $T \in \mathcal{T}_{x}$ only meets $\lesssim_{D} 1$ pre-images $\pi_{x}^{-1}(I)$, $I \in \mathcal{I}_{x}$. This yields the following upper bound, which will be needed somewhat later:
\begin{equation}\label{form19} \mu(K_{x} \cap T) \lesssim_{D} \max_{I \in \mathcal{I}_{x}} (\mu|_{K_{x}})(\pi_{x}^{-1}(I)) \stackrel{\eqref{form18}}{=} \max_{I \in \mathcal{I}_{x}} \mu_{x}'(I) \stackrel{\eqref{form13}}{\leq} \Delta_{x}^{\Sigma_{x}}, \qquad T \in \mathcal{T}_{x}, \, x \in E. \end{equation}
We note that $\mathcal{T}_{x}$ is a
$(\Delta_{x},\Sigma_{x},\Delta_{x}^{-3\epsilon})$-set of
$O_{D}(\Delta_{x})$-tubes containing $x$, if $\delta > 0$ is small
enough. This is just a restatement of Claim \ref{SClaim}. We will
fix later the small issue that $\mathcal{T}_{x}$ consists of
$O_{D}(\Delta_{x})$-tubes instead of $\Delta_{x}$-tubes.

Our next task will be to eliminate the $x$-dependence of the scale $\Delta_{x}$, and the exponent $\Sigma_{x}$, by further pigeonholing. Recall that $\Delta_{x} \in 2^{-\N}$, and $\delta \leq \Delta_{x} \leq \delta^{\epsilon^{3/\epsilon}}$ by \eqref{form9}. Also, recall that the choice of the initial scale $\delta = \delta_{x} > 0$ was common for all $x \in E_{0} \subset E$, where $\nu(E_{0}) \gtrsim (\log (1/\delta))^{-2}$. Therefore, by the pigeonhole principle, there exists a fixed dyadic scale $\Delta \in [\delta,\delta^{\epsilon^{3/\epsilon}}]$, and a subset $E' \subset E_{0} \subset \spt \nu$ with the properties
\begin{equation}\label{form53} \nu(E') \gtrsim (\log \tfrac{1}{\delta})^{-4} \geq \epsilon^{12/\epsilon} (\log \tfrac{1}{\Delta})^{-1} \quad \text{and} \quad \Delta_{x} = \Delta \text{ for all } x \in E'. \end{equation}
Now, even though $\Delta_{x} = \Delta$ is common for all $x \in E'$, the exponents $\Sigma_{x} \in [0,1]$ still depend on $x$. To fix this, we use the pigeonhole principle to pick a further subset $E'' \subset E'$ of measure $\nu(E'') \geq \epsilon \cdot \nu(E')$, and a number $\Sigma \in [0,1]$ (actually $\Sigma \leq \sigma + \epsilon$) such that
\begin{equation}\label{form15} \Sigma - \epsilon \leq \Sigma_{x} \leq \Sigma, \qquad x \in E''. \end{equation}
Evidently $\nu(E'') \gtrsim_{\epsilon} (\log (1/\Delta))^{-1}$. To simplify notation in the sequel, we keep denoting $E''$ by $E$; we just have to keep in mind that $\nu(E) \gtrsim_{\epsilon} (\log (1/\Delta))^{-1}$ with the new notation. With the new notation, and recalling \eqref{form13}, we find
\begin{displaymath} (\pi_{x}\mu)(I) \geq \mu_{x}'(I) \geq \Delta_{x}^{\Sigma_{x} + \eta} \geq \Delta^{\Sigma + \epsilon}, \qquad x \in E, \, T \in \mathcal{I}_{x}. \end{displaymath}
In particular, the tubes $T \in \mathcal{T}_{x}$ (each of which contained a set of the form $K \cap \pi_{x}^{-1}(I)$) satisfy
\begin{equation}\label{form14} \mu(T) \geq \Delta^{\Sigma + \epsilon}, \qquad x \in E, \, T \in \mathcal{T}_{x}. \end{equation}
We recall that $\mathcal{T}_{x}$, $x \in E$, was already shown to be a $(\Delta_{x},\Sigma_{x},\Delta_{x}^{-3\epsilon})$-set. Since $\Delta_{x} = \Delta$ and $\Sigma_{x} \in (\Sigma - \epsilon,\Sigma)$ for all $x \in E$  (by \eqref{form15}), it follows that $\mathcal{T}_{x}$ is a $(\Delta,\Sigma,\Delta^{-4\epsilon})$-set.

We then begin to seek the sets $P_{E}$ and $P_{K}$ appearing in \eqref{i}--\eqref{iii}. We start with $P_{E}$. By one more application of the pigeonhole principle, the set
\begin{displaymath} E^{v} := \{x \in E : v/2 \leq \nu(B(x,\Delta)) \leq v\} \end{displaymath}
satisfies $\nu(E^{v}) \approx_{\Delta,\epsilon} 1$ for some dyadic rational $v \leq \Delta^{s}$. Here the "$\approx_{\Delta}$" notation hides constants of the form $C(\log (1/\Delta))^{C}$, where $C = C(\|\nu\|,\epsilon)$. We let $P_{E} \subset E^{v}$ be a maximal $\Delta$-separated set. Then evidently $|P_{E}| \approx_{\Delta,\epsilon} v^{-1}$, and
\begin{equation}\label{form6} |P_{E} \cap B(x,r)| \lesssim \frac{\nu(B(x,2r))}{v} \lessapprox_{\Delta,\epsilon} r^{s}|P_{E}|, \qquad x \in \R^{2}, \, r \geq \Delta. \end{equation}
Therefore, if $\delta > 0$ is small enough, we see that $P_{E} \subset E$ is a non-empty $(\Delta,s,\Delta^{-\epsilon})$-set.

Finding the set $P_{K} \subset K$ proceeds similarly, although
with one small additional complication. Namely, we start by
considering subsets of $K_{x}$, instead of $K$ directly. Here
$K_{x}$ was the set defined in \eqref{form12}, with $\mu(K_{x})
\geq \delta^{2\eta}$. For every $x \in P_{E}$, we choose (by the
pigeonhole principle) a dyadic rational $w(x) \leq \Delta^{t}$
with the property that the set
\begin{displaymath} K_{x}^{w(x)} := \{x \in K_{x} : w(x)/2 \leq \mu(B(x,\Delta)) \leq w(x)\} \end{displaymath}
satisfies
\begin{equation}\label{form35} \mu(K_{x}^{w(x)}) \gtrapprox_{\Delta} \mu(K_{x}) \geq \delta^{2\eta} \stackrel{\eqref{form9}}{\geq} \Delta^{2\eta \cdot \epsilon^{-3/\epsilon}} \stackrel{\eqref{eta}}{\geq} \Delta^{\epsilon}. \end{equation}
Next, we wish to make $w(x)$ independent of $x \in P_{E}$. This is achieved by more pigeonholing: there exists a dyadic rational $w \leq \Delta^{t}$, and a subset $P_{E}' \subset P_{E}$ with $|P_{E}'| \approx_{\Delta} |P_{E}|$ such that $w(x) = w$ for all $x \in P_{E}'$. Then $P_{E}'$ is a $(\Delta,s,\Delta^{-2\epsilon})$-set, if $\delta > 0$ is small enough, and we redefine $P_{E} := P_{E}'$. The set "$P_{E}$" is the final set appearing in the statement of the proposition.

Finally, we let $P_{K} \subset K$ be a maximal $\Delta$-separated set in
\begin{displaymath} K^{w} := \{x \in K : w/2 \leq \mu(B(x,\Delta)) \leq w\}. \end{displaymath}
Evidently $K^{w} \supset K_{x}^{w}$ for every $x \in P_{E}$, so $\mu(K^{w}) \geq \mu(K_{x}^{w}) \geq \Delta^{2\epsilon}$ for every $x \in P_{E}$, by \eqref{form35}. Now
\begin{equation} \label{Pk}
w^{-1}\Delta^{2\epsilon} \lesssim |P_{K}| \lesssim w^{-1},
\end{equation} and a calculation similar to that in \eqref{form6}, now using the inequality $\mu(B(x,r)) \leq r^{t}$, shows that $P_{K}$ is a non-empty $(\Delta,t,\Delta^{-3\epsilon})$-set.

What remains is to conclude the proof of property \eqref{iii} in the statement of Proposition \ref{prop2} -- to demonstrate that $|T \cap P_{K}| \geq \Delta^{\Sigma + O(\epsilon)}|P_{K}|$ for all $T \in \mathcal{T}_{x}$, and for all $x \in P_{E}$. To be accurate, this will be true for all $T$ in a large sub-family $\mathcal{T}_{x}' \subset \mathcal{T}_{x}$.

Fix $x \in P_{E}$. We recall from \eqref{form19} that $\mu(K_{x} \cap T) \lesssim_{D} \Delta_{x}^{\Sigma_{x}} \leq \Delta^{\Sigma - \epsilon}$ for all $T \in \mathcal{T}_{x}$, for $x \in P_{E} \subset E$. (It seems difficult to maintain this upper bound for $K \cap T$ instead of $K_{x} \cap T$, and this is the main reason for working with the sets "$K_{x}$" all the way down here.) Moreover, recall from \eqref{form36} that $|\mathcal{T}_{x}| \lesssim_{D} \Delta_{x}^{-\Sigma_{x} - \eta}$, and consequently $|\mathcal{T}_{x}| \leq \Delta^{-\Sigma - \epsilon}$ for all $x \in E$, and $\delta > 0$ small enough. Since the sets $K_{x} \cap T$, $T \in \mathcal{T}_{x}$, cover $K_{x}^{w}$, and $\mu(K_{x}^{w}) \geq \Delta^{2\epsilon}$, it follows that there exists a subset $\mathcal{T}_{x}' \subset \mathcal{T}_{x}$ of cardinality $|\mathcal{T}_{x}'| \geq \Delta^{-\Sigma + \epsilon} \geq \Delta^{2\epsilon}|\mathcal{T}_{x}|$ such that
\begin{equation}\label{form7}
\mu(K^{w}\cap T) \geq \mu(K_{x}^{w}\cap T) \geq \Delta^{\Sigma + 3\epsilon}, \qquad T \in \mathcal{T}_{x}'.
\end{equation}
Since $\mathcal{T}_{x}$ was a $(\Delta,\Sigma,\Delta^{-4\epsilon})$-set, $\mathcal{T}_{x}'$ is a $(\Delta,\Sigma,\Delta^{-6\epsilon})$-set. Finally, note that by \eqref{form7}, and the definitions of $K^{w}$ and $P_{K}$, every tube $2T$ with $T \in \mathcal{T}_{x}'$ satisfies
\begin{displaymath}
w \cdot |2T \cap P_{K}| \gtrsim \mu(K^{w}\cap T) \geq \Delta^{\Sigma + 3\epsilon} \quad \Longrightarrow \quad |2T \cap P_{K}| \gtrsim \Delta^{\Sigma + 3\epsilon}/w \overset{\eqref{Pk}}{\gtrsim} \Delta^{\Sigma + 3\epsilon}|P_{K}|.
\end{displaymath}
Therefore the family $\{2T : T \in \mathcal{T}_{x}'\}$ satisfies property \eqref{iii} of Proposition \ref{prop2}, except that the tubes in this family have thickness $O_{D}(\Delta)$ instead of $\Delta$. However, at this point it is simple to cover each tube $2T$, $T \in \mathcal{T}_{x}'$, by $\sim_{D} 1$ tubes of width $\Delta$, and pick one of them, say $T'$, which still satisfies $|T' \cap P_{K}| \geq \Delta^{\Sigma + 4\epsilon}|P_{K}|$ (for $\delta > 0$ small enough). The family of $\Delta$-tubes $T'$ so obtained satisfies property \eqref{iii}, and this concludes the proof the proposition. \end{proof}

\section{Proof of Theorem \ref{main2}}\label{s:discrete}

In this section, we complete the proof of Theorem \ref{main2}. The result will be derived from a recent incidence theorem of
Fu and Ren \cite[Theorem 4.8]{2021arXiv211105093F} concerning $(\delta,s)$-sets of points and $(\delta,t)$-sets of tubes. The theorem of Fu and Ren is formulated in terms of a slightly different (and more classical) notion of $(\delta,s)$-sets. We start by stating this definition, and then we explore the connection to (our) $(\delta,s)$-sets.

\begin{definition}[Katz-Tao $(\delta,s)$-set] Let $(X,d)$ be a metric space. We say that a $\delta$-separated set $P \subset X$ is a \emph{Katz-Tao $(\delta,s,C)$-set} if
\begin{displaymath} |P \cap B(x,r)| \leq C\left(\frac{r}{\delta} \right)^{s}, \qquad x \in \R^{d}, \, r \geq \delta. \end{displaymath}
\end{definition}

As the name suggests, the Katz-Tao $(\delta,s)$-sets were introduced by Katz and Tao \cite{MR1856956}. The next lemma shows that $(\delta,t)$-sets can be decomposed into Katz-Tao $(\delta,t)$-sets:

\begin{lemma}\label{lemma3} Let $(X,d)$ be a doubling metric space with constant $D \geq 1$.\footnote{Every ball of radius $r > 0$ can be covered by $D \in \N$ balls of radius $r/2$, with $D$ independent of $r$.} For every $\epsilon,t > 0$, there exists $\delta_{0} = \delta_{0}(\epsilon,D,t) > 0$ such that the following holds for all $\delta \in (0,\delta_{0}]$. Let $P \subset B(x_{0},1) \subset X$ be a $\delta$-separated $(\delta,t,C)$-set. Then $P$ can be written as a disjoint union
\begin{displaymath} P = P_{1} \cup \ldots \cup P_{N} \end{displaymath}
where each $P_{j}$ is a Katz-Tao $(\delta,t,1)$-set, and $N \leq C|P|\delta^{t-\epsilon}$.
\end{lemma}

\begin{proof} The method is the same as in the proof of \cite[Proposition 4.5]{2021arXiv211105093F}. For each dyadic rational $r \in [\delta,2]$, let $\mathcal{B}_{r}$ be a cover of $B(x_{0},1)$ by balls of radius $r$ with overlap $\lesssim_{D} 1$, and with the property that every ball of radius $\leq r/2$ intersecting $B(x_0,1)$ is contained in at least one of the balls in $\mathcal{B}_{r}$. To construct $\mathcal{B}_{r}$, choose a maximal $(r/2)$-separated set $X_{r} \subset B(x_0,1)$, and set $\mathcal{B}_{r} := \{B(x,r) : x \in X_{r}\}$. The bounded overlap of $\mathcal{B}_{r}$ follows from the doubling hypothesis of $X$. Indeed, if some $x \in B(x_0,1)$ lies in $N$ distinct of $\mathcal{B}_{r}$, then the centres of these balls form an $(r/2)$-separated subset of $B(x,2r)$ of cardinality $N$. By the doubling hypothesis, $N \lesssim_{D} 1$.

We then begin the proof in earnest. Assume that $P \neq \emptyset$, otherwise there is nothing to prove. Set
\begin{equation}\label{defH} H := 4^{t+1} C|P|\delta^{t} \geq 1. \end{equation}
The lower bound "$\geq 1$" follows from the assumption that $P \neq \emptyset$ is a $(\delta,t,C)$-set. Fix $r \in 2^{-\N}$ with $r \in [\delta,1]$ and $B \in \mathcal{B}_{r}$. We divide the points in $P \cap B$ into $m(B) = \ceil{|P \cap B|/H}$ groups $G^{r}_{1}(B),\ldots,G^{r}_{m(B)}(B)$ by forming as many groups of size exactly "$H$" as possible, and then one remainder group of size $\leq H$. It is of course possible that $|P \cap B| < H$: in this case $m(B) = 1$, and we only have the remainder group.

We then form a graph $G = (P,E)$, whose edge set $E \subset P \times P$ is defined as follows. For every group $G_{j}^{r}(B)$, $1 \leq j \leq m(B)$, we connect all the members of the group to each other by an edge. Then, we do this for every $B \in \mathcal{B}_{r}$, and for every dyadic rational $r \in [\delta,1]$.

What is the maximum degree of $G$? For every $x \in P$ and $r \in [\delta,1]$ fixed, $x$ is connected to every other point in its own group $G^{r}_{j}(B)$. This statement holds whenever $B \in \mathcal{B}_{r}$ contains $x$, and this may happen for $\lesssim_{D} 1$ different choices $B \in \mathcal{B}_{r}$. So, every $r \in [\delta,1]$ yields $\lesssim_{D} H$ edges incident to $x$. The number of dyadic scales $r \in [\delta,1]$ is $\sim \log (1/\delta)$, so $\max_{x \in P} \deg_{G}(x) \lesssim_{D} H\log (1/\delta)$.

As in \cite[Lemma 4.7]{2021arXiv211105093F}, we may now deduce from Brook's theorem (see \cite{MR12236} or \cite{MR396344}) that the graph $G$ admits a colouring of the vertices $P$ with $N \lesssim_{D} H\log(1/\delta)$ colours with the property that no two adjacent vertices share the same colour. The colouring induces a partition $P = P_{1} \cup \ldots \cup P_{N}$, where
\begin{displaymath} N \lesssim_{D} H\log(1/\delta) \stackrel{\eqref{defH}}{=} 4^{t+1} C|P|\delta^{t} \log(1/\delta). \end{displaymath}
In particular, $N \leq C|P|\delta^{t - \epsilon}$ if $\delta > 0$ is small enough, depending only on $\epsilon,D,t$.

It remains to check that each $P_{j}$ is a Katz-Tao $(\delta,t,1)$-set. Fix $1 \leq j \leq N$. Let $x\in X$ and $\delta \leq \rho \leq 1$. If $B(x,\rho) \cap B(x_{0},1) = \emptyset$, then $|P_{j} \cap B(x,\rho)| = 0$. Otherwise $B(x,\rho)$ is contained in one of the balls $B \in \mathcal{B}_{r}$ for some $r/4 \leq \rho \leq r/2$. By the $(\delta,t,C)$-property of $P$, we have
\begin{displaymath} |P \cap B| \leq Cr^{t}|P| \leq 4^{t}(C\rho^{t}|P|). \end{displaymath}
This implies that the number "$m(B)$" of groups $G_{1}^{r}(B),\ldots,G_{m(B)}^{r}(B)$ is at most
\begin{displaymath} m(B) = \ceil{|P \cap B|/H} \leq \ceil{4^{t}4^{-t-1}(\rho/\delta)^{t}} \leq (\rho/\delta)^{t}. \end{displaymath}
 The set $P_{j}$ contains at most one point in each of these groups, therefore $|P_{j} \cap B(x,\rho)| \leq |P_{j} \cap B| \leq (\rho/\delta)^{t}$. This completes the proof of the lemma. \end{proof}




We next state the incidence theorem of Fu and Ren \cite[Theorem 4.8]{2021arXiv211105093F}. Given a set of points $P$ and a set of tubes $\mathcal{T}$,
we let $\mathcal{I}(P,\mathcal{T}) := \{(p,T) : p \in T\}$ be the set of \emph{incidences} between $P$ and $\mathcal{T}$.
 The original version was stated for Katz-Tao $(\delta,s)$-sets, but the version below will follow by combining the original statement with Lemma \ref{lemma3}:
\begin{thm}\label{FuRen} Let $0 \leq s,t \leq 2$. Then, for every $\epsilon > 0$, there exist $\delta_{0} = \delta_{0}(\epsilon) > 0$ such that the following holds for all $\delta \in (0,\delta_{0}]$. If $P \subset B(1)$ is a  $\delta$-separated $(\delta,s,\delta^{-\epsilon})$-set, and $\mathcal{T}$ is a $\delta$-separated $(\delta,t,\delta^{-\epsilon})$-set, then
\begin{displaymath} |\mathcal{I}(P,\mathcal{T})| \leq |P||\mathcal{T}| \cdot \delta^{\kappa(s + t - 1) - 5\epsilon}, \end{displaymath}
where $\kappa = \kappa(s,t) = \min\{1/2,1/(s + t - 1)\}$. \end{thm}

\begin{proof}[Proof of Theorem \ref{FuRen}] By Lemma \ref{lemma3} applied in both $\R^{2}$ and $\mathcal{A}(2,1)$, we may write
\begin{displaymath} P = P_{1} \cup \ldots \cup P_{M} \quad \text{and} \quad \mathcal{T} = \mathcal{T}_{1} \cup \ldots \cup \mathcal{T}_{N}, \end{displaymath}
where $M \leq |P|\delta^{s - 2\epsilon}$ and $N \leq |\mathcal{T}|\delta^{t - 2\epsilon}$, each $P_{j}$ is a Katz-Tao $(\delta,s,1)$-set, and each $\mathcal{T}_{j}$ is a Katz-Tao $(\delta,t,1)$-set.
 By the original version of \cite[Theorem 4.8]{2021arXiv211105093F}, we have
\begin{displaymath} |\mathcal{I}(P_{i},\mathcal{T}_{j})| \leq \delta^{-s - t} \cdot \delta^{\kappa(s + t - 1) - \epsilon}, \qquad 1 \leq i \leq M, \, 1 \leq j \leq N, \end{displaymath}
assuming that $\delta = \delta(\epsilon) > 0$ is small enough. Therefore,
\begin{displaymath} |\mathcal{I}(P,\mathcal{T})| \leq \sum_{i = 1}^{M} \sum_{j = 1}^{N} |\mathcal{I}(P_{i},\mathcal{T}_{j})| \leq MN \cdot \delta^{-s - t} \cdot \delta^{\kappa(s + t - 1) - \epsilon} \leq |P||\mathcal{T}| \cdot \delta^{\kappa(s + t - 1) - 5\epsilon}. \end{displaymath}
This concludes the proof. \end{proof}

The next lemma allows us to find $(\delta,s)$-sets inside $\delta$-discretised Furstenberg sets.

\begin{lemma}\label{lemma4} Let $s \in [0,1]$ and $t \in [0,2]$. Then, for every $\zeta > 0$, there exists $\delta_{0} = \delta_{0}(s,t,\zeta) > 0$ and $\epsilon = \epsilon(s,t,\zeta) > 0$ such that the following holds for all $\delta \in (0,\delta_{0}]$. Assume that $\mathcal{T}$ is a non-empty $(\delta,t,\delta^{-\epsilon})$-set of $\delta$-tubes in $\R^{2}$. Assume that for every $T \in \mathcal{T}$ there exists a non-empty $(\delta,s,\delta^{-\epsilon})$-set $P_{T} \subset T \cap B(1)$.
Then, the union
\begin{equation}\label{form32a} P := \bigcup_{T \in \mathcal{T}} P_{T} \end{equation}
contains a non-empty $(\delta,\gamma(s,t),\delta^{-\zeta})$-set, where
\begin{equation}\label{form31} \gamma(s,t) = s + \min\{s,t\}. \end{equation}
\end{lemma}

\begin{remark} In the range $t > s$, the conclusion of Lemma \ref{lemma4} could be slightly improved: we could take $\gamma(s,t) = 2s + \epsilon(s,t)$, as shown in \cite{2021arXiv210603338O}. For $t > 1$ or $t > 2s$, further improvements are possible, see the introduction to \cite{2021arXiv210704471D} for an account of the best current results. However, Theorem \ref{main2} (hence Theorem \ref{main}) already follows from Lemma \ref{lemma4} as stated, and the sharper bounds for $\gamma(s,t)$ would not result in any further improvements. The reason is that the main application of Lemma \ref{lemma4} in the proof of Theorem \ref{main} occurs in the case $t \leq s$, and in this regime the bound $\gamma(s,t) = s + t$ is sharp.
\end{remark}

\begin{proof}[Proof of Lemma \ref{lemma4}] We only sketch the argument, since it is nearly follows from existing statements, and the full details are very standard (if somewhat lengthy). The main point is the following: it is known that the Hausdorff dimension of every $(s,t)$-Furstenberg set $F \subset \R^{2}$ satisfies $\Hd F \geq \gamma(s,t)$, where $\gamma(s,t)$ is the function defined in \eqref{form31}. The case $t \leq s$ is due to Lutz and Stull \cite{MR4179019}; they used information theoretic methods, but a more classical proof is also available, see \cite[Theorem A.1]{HSY21}. The case $t \geq s$ essentially goes back to Wolff in \cite{Wolff99}, but also literally follows from \cite[Theorem A.1]{HSY21}.

While the statement in \cite[Theorem A.1]{HSY21} only concerns Hausdorff dimension, the proof goes via Hausdorff content,
 and the following statement can be extracted from the argument. Let $P$ be the set defined in \eqref{form32a}. Then, the $\gamma(s,t)$-dimensional Hausdorff content of the $\delta$-neighbourhood $P(\delta)$ satisfies
\begin{equation}\label{form33} \mathcal{H}_{\infty}^{\gamma(s,t)}(P(\delta)) \geq \delta^{\zeta}, \end{equation}
assuming that $\epsilon = \epsilon(s,t,\zeta) > 0$ and the upper bound $\delta_{0} = \delta(s,t,\zeta) > 0$ for the scale $\delta$ were chosen small enough. The claim in the lemma immediately follows from \eqref{form33}, and \cite[Proposition A.1]{FasslerOrponen14}. This proposition, in general, states that if $B \subset \R^{d}$ is a set with $\mathcal{H}^{s}_{\infty}(B) = \kappa > 0$, then $B$ contains a non-empty $(\delta,s,C\kappa^{-1})$-set for some absolute constant $C > 0$. In particular, from \eqref{form33} we see that $P(\delta)$ contains a $(\delta,\gamma(s,t),C\delta^{-\zeta})$-set. This easily implies a similar conclusion about $P$ itself. \end{proof}

By standard point-line duality considerations (see a few details below the statement), Lemma \ref{lemma4} is equivalent to the following statement concerning tubes:
\begin{lemma}\label{lemma5} Let $s \in [0,1]$ and $t \in [0,2]$. Then, for every $\zeta > 0$, there exists $\delta_{0} = \delta_{0}(s,t,\zeta) > 0$ and $\epsilon = \epsilon(s,t,\zeta) > 0$ such that the following holds for all $\delta \in (0,\delta_{0}]$. Assume that $P \subset B(1) \subset \R^{2}$ is a non-empty $(\delta,t,\delta^{-\epsilon})$-set. Assume that for every $x \in P$ there exists a non-empty $(\delta,s,\delta^{-\epsilon})$-set of tubes $\mathcal{T}_{x}$ with the property that $x \in T$ for all $T \in \mathcal{T}_{x}$. Then, the union
\begin{equation}\label{form32} \mathcal{T} := \bigcup_{x \in P} \mathcal{T}_{x} \end{equation}
contains a non-empty $(\delta,\gamma(s,t),\delta^{-\zeta})$-set, where $\gamma(s,t) = s + \min\{s,t\}$. \end{lemma}

If the reader is not familiar with point-line duality, then the full details in a very similar context are recorded in \cite[Sections 6.1-6.2]{2021arXiv210704471D}. Here we just describe the key ideas. To every point $(a,b) \in \R^{2}$, we associate the line $\mathbf{D}(a,b) := \{y = ax + b : x \in \R\} \in \mathcal{A}(2,1)$. Conversely, to every line $\ell = \{y = cx + d : x \in \R\}$ we associate the point $\mathbf{D}^{\ast}(\ell) = (-c,d)$.
Then, it is easy to check that
\begin{equation}\label{form34} p \in \ell \quad \Longleftrightarrow \quad \mathbf{D}^{\ast}(\ell) \in \mathbf{D}(p). \end{equation}
For $(a,b),(c,d) \in [0,1]^{2}$, say, the maps $\mathbf{D}$ and $\mathbf{D}^{\ast}$ are bilipschitz between the Euclidean metric, and the metric on $\mathcal{A}(2,1)$. Therefore the property of "being a $(\delta,s)$-set" is preserved (up to inflating the constants slightly). Now, roughly speaking, Lemma \ref{lemma5} follows from Lemma \ref{lemma4} by first applying the transformations $\mathbf{D},\mathbf{D}^{\ast}$ to the points $P$ and the tubes $\mathcal{T}_{x}$, $x \in P$, respectively. The main technicalities arise from the fact that $\mathcal{T}_{x}$ is a set of $\delta$-tubes, and not a set of lines. Let us ignore this issue for now, and assume that $\mathcal{T}_{x} = \mathcal{L}_{x}$ is actually a $(\delta,s)$-set of lines such that $x \in \ell$ for all $\ell \in \mathcal{L}_{x}$. In this case Lemma \ref{lemma5} is simple to infer from Lemma \ref{lemma4}.

Write $P = \mathbf{D}^{\ast}(\mathcal{L})$ for some $(\delta,t)$-set of lines $\mathcal{L} \subset \mathcal{A}(2,1)$, and write also $\mathcal{L}_{x} = \mathbf{D}(P_{x})$ for some $(\delta,s)$-set of points $P_{x} \subset \R^{2}$. Now, if $\ell \in \mathcal{L}$, then $\mathbf{D}^{\ast}(\ell) = x \in \mathbf{D}(y)$ for all $y \in P_{x}$ by assumption. By \eqref{form34}, this is equivalent to $P_{x} \subset \ell$. Thus, every line $\ell = (\mathbf{D}^{\ast})^{-1}(x) \in \mathcal{L}$, $x \in P$, contains a $(\delta,s)$-set $P_{x} =: P_{\ell}$. This places us in a position to apply Lemma \ref{lemma4}.

A similar argument still works if $\mathcal{L}_{x}$ is replaced by the $(\delta,s)$-set of tubes $\mathcal{T}_{x}$. One only needs to make sure that if $x \in T \in \mathcal{T}_{x}$, then the line $\ell = (\mathbf{D}^{\ast})^{-1}(x)$ is $O(\delta)$-close to a certain $(\delta,s)$-set $P_{\ell}$; this set can be derived from $\mathcal{T}_{x}$ by using the idea above. For the technical details, we refer to \cite[Sections 6.1-6.2]{2021arXiv210704471D}, in particular \cite[Lemma 6.7]{2021arXiv210704471D}.

We are finally equipped to prove Theorem \ref{main2}:

\begin{proof}[Proof of Theorem \ref{main2}]
Fix $s \in [0,2]$, $t \in (1,2]$, $\sigma \in [0,1)$, and $\eta > 0$. Let $P_{K},P_{E} \subset B(1) \subset \R^{2}$ be as in the statement of the theorem: thus $P_{K}$ is a $(\delta,t,\delta^{-\epsilon})$-set, and $P_{E}$ is a $(\delta,s,\delta^{-\epsilon})$-set. Recall also the $(\delta,\sigma,\delta^{-\epsilon})$-sets of tubes $\mathcal{T}_{x}$ passing through $x$, for every $x \in P_{E}$, with the property
\begin{equation}\label{form28} |P_{K} \cap T| \geq \delta^{\sigma + \epsilon}|P_{K}|, \quad T \in \mathcal{T}_{x}. \end{equation}
The claim is that
\begin{equation}\label{form30} s \leq \max\{1 + \sigma - t + \eta,0\} \end{equation}
if $\delta,\epsilon$ are chosen small enough, depending on $t,\sigma,\eta$.

 By Lemma \ref{lemma5}, the union $\bigcup_{x \in P_{E}} \mathcal{T}_{x}$ contains a non-empty $(\delta,\gamma(\sigma,s),\delta^{-\zeta})$-set $\mathcal{T}$, where
 \begin{displaymath} \gamma(\sigma,s) = \sigma + \min\{s,\sigma\}, \end{displaymath}
 and $\zeta > 0$ can be made as small as we like by choosing $\epsilon,\delta > 0$ sufficiently small. In particular, we will require that $10\zeta + 2\epsilon \leq \eta$. We may assume that $\zeta \geq \epsilon$, so our $(\delta,t,\delta^{-\epsilon})$-set $P_{K}$ is also a $(\delta,t,\delta^{-\zeta})$-set (if $\zeta < \epsilon$, then both $P_{K}$ and $\mathcal{T}$ are $(\delta,u,C)$-sets with constant $C = \delta^{-\epsilon}$, and this would work even better in the sequel).

 By \eqref{form28}, we have
 \begin{equation}\label{form29} |P_{K}||\mathcal{T}| \cdot \delta^{\sigma + \epsilon} \leq \sum_{T \in \mathcal{T}} |P_{K} \cap T| = |\mathcal{I}(P_{K},\mathcal{T})|. \end{equation}
 We next compare this lower bound for $|\mathcal{I}(P_{K},\mathcal{T})|$ against the upper bounds from Theorem \ref{FuRen}. Recall the exponent "$\kappa$" from the statement of Theorem \ref{FuRen}. Since $P_{K}$ is a $(\delta,t)$-set and $\mathcal{T}$ is a $(\delta,\gamma(\sigma,s))$-set, the useful quantity for us is
  \begin{displaymath} \bar{\kappa}(s,\sigma,t) := \kappa(t,\gamma(\sigma,s)) = \min\{1/2,1/(t + \gamma(\sigma,s) - 1)\}. \end{displaymath}
The remainder of the proof splits into four cases:
 \begin{enumerate}
 \item Assume first that $s \leq \sigma$. Thus $\gamma(\sigma,s) = s + \sigma$, so $\mathcal{T}$ is a $(\delta,s + \sigma,\delta^{-\zeta})$-set.
 \begin{itemize}
  \item[(a)] Assume that $\bar{\kappa}(s,\sigma,t) = 1/2$. Then, by Theorem \ref{FuRen},
  \begin{displaymath} |\mathcal{I}(P_{K},\mathcal{T})| \leq |P_{K}||\mathcal{T}| \cdot \delta^{(t + (s + \sigma) - 1)/2 - 5\zeta}. \end{displaymath}
  Comparing this against \eqref{form29} yields $\delta^{2\sigma+2\epsilon} \leq \delta^{t + s + \sigma - 1 - 10\zeta}$, and therefore $s \leq 1 + \sigma - t + 10\zeta + 2\epsilon$. This yields \eqref{form30}, since we assumed that $10\zeta + 2\epsilon \leq \eta$.
  \item[(b)] Assume that $\bar{\kappa}(s,\sigma,t) = 1/(t + s + \sigma - 1)$. Then,
  \begin{displaymath} |\mathcal{I}(P_{K},\mathcal{T})| \leq |P_{K}||\mathcal{T}| \cdot \delta^{1 - 5\zeta}. \end{displaymath}
  Comparing this against \eqref{form29} yields $\delta^{\sigma} \leq \delta^{1 - 5\zeta - \epsilon}$. Since $\sigma < 1$, this scenario is not possible for $\delta,\epsilon$ (hence also $\zeta$) small enough, depending only on $\sigma$.
 \end{itemize}
 \item Assume second that $s > \sigma$. Thus $\gamma(\sigma,s) = 2\sigma$, so $\mathcal{T}$ is a $(\delta,2\sigma,\delta^{-\zeta})$-set.
 \begin{itemize}
 \item[(a)] Assume that $\bar{\kappa}(s,\sigma,t) = \tfrac{1}{2}$. Then,
 \begin{displaymath} |\mathcal{I}(P_{K},\mathcal{T})| \leq |P_{K}||\mathcal{T}| \cdot \delta^{(t + 2\sigma - 1)/2 - 5\zeta}. \end{displaymath}
 Comparing this against \eqref{form29} yields $1 \leq \delta^{(t - 1)/2 - 5\zeta - \epsilon}$. Since $t > 1$, this scenario is not possible for $\delta,\epsilon$ (hence also $\zeta$) small enough, depending on $t$.
 \item[(b)] Assume finally that $\bar{\kappa}(s,\sigma,t) = 1/(t + 2\sigma - 1)$. Then,
 \begin{displaymath} |\mathcal{I}(\mathcal{P}_{K},\mathcal{T})| \leq |P_{K}||\mathcal{T}| \cdot \delta^{1 - 5\zeta}. \end{displaymath}
 As in case (1)(b) above, this leads to the impossible situation $\delta^{\sigma} \leq \delta^{1 - 5\zeta - \epsilon}$.
 \end{itemize}
 \end{enumerate}
 We have now seen that the cases (1)(b) and (2)(a)-(b) are not possible for $\delta,\epsilon$ small enough, depending only on $\sigma < 1$ and $t > 1$. Case (1)(a), on the other hand, yields the desired inequality \eqref{form30} for $10\zeta + 2\epsilon \leq \eta$. This completes the proof of Theorem \ref{main2}. \end{proof}

 \section{Proof of Theorem \ref{main3}}\label{s:secondMain}

The purpose of this section is to prove Theorem \ref{main3}, restated below:
\begin{thm}\label{main4} Let $K \subset \R^{2}$ be a Borel set with $\Hd K \leq 1$. Then,
\begin{displaymath} \Hd \{x \in \R^{2} \, \setminus \, K : \Hd \pi_{x}(K) < \Hd K\} \leq 1. \end{displaymath}
\end{thm}

 In this section, we use the notation
 \begin{displaymath} \mathcal{H}^{s}_{\delta,\infty}(K) := \inf \Big\{\sum_{j} \diam(E_{i})^{s} : K \subset \bigcup_{j} E_{j} \Big\}, \qquad K \subset \R^{d}, \end{displaymath}
 where the "$\inf$" runs over all countable covers of $K$ with sets $E_{j} \subset \R^{d}$ with $\diam(E_{j}) \geq \delta$. Theorem \ref{main4} will be deduced as a corollary of a $\delta$-discretised version of Theorem \ref{main}:

 \begin{proposition}\label{prop4} Let $0 < \sigma \leq s \leq 1$, $t \in (1,2]$, $\eta \in (0,1]$, and $s > \max\{1 + \sigma - t,0\}$. Then, there exist
 \begin{displaymath} \delta_{0} = \delta(s,\sigma,t,\eta) > 0 \quad \text{and} \quad \epsilon = \epsilon(s,\sigma,t,\eta) > 0 \end{displaymath}
 such that the following holds for all $\delta \in (0,\delta_{0}]$.

 Let $E,F \subset B(1) \subset \R^{2}$ be non-empty $\delta$-separated sets, where $E$ is a $(\delta,s,\delta^{-\epsilon})$-set, $F$ is a $(\delta,t,\delta^{-\epsilon})$-set, and $\dist(E,F) \geq \tfrac{1}{2}$. Then, there exists $y \in E$ such that
 \begin{displaymath} \mathcal{H}^{\sigma}_{\delta,\infty}(\pi_{y}(F')) \geq \delta^{\eta}, \qquad\text{for all } F' \subset F, \, |F'| \geq \delta^{\epsilon}|F|. \end{displaymath}
 \end{proposition}


 \begin{proof}[Proof of Proposition \ref{prop4}] The proof of Proposition \ref{prop4} is very similar to the proof of Theorem \ref{main}, so we only sketch the argument. Assume to the contrary that for every $y \in E$, there exists a subset $F_{y}' \subset F$ with
 \begin{equation}\label{form52} |F_{y}'| \geq \delta^{\epsilon}|F| \quad \text{and} \quad \mathcal{H}^{\sigma}_{\delta,\infty}(\pi_{y}(F_{y}')) < \delta^{\eta}. \end{equation}
 This allows us to run the proof of Proposition \ref{prop2} almost verbatim, and find the objects whose existence was guaranteed by that proposition. To see a few details, consider the unit-normalised counting measures $\mu$ and $\nu$ on $F$ and $E$, respectively. These measures play exactly the same role as the measures "$\mu$" and "$\nu$" in the proof of Proposition \ref{prop2}. Now, \eqref{form52}, combined with the pigeonhole principle, allows us to find a scale $\delta_{1} \in [\delta,\delta^{\eta/\sigma}]$, and a collection of dyadic arcs $\mathcal{J}_{y} \subset \mathcal{D}_{\delta_{1}}(S^{1})$ with the properties
 \begin{equation}\label{form54} |\mathcal{J}_{y}| \leq \delta_{1}^{-\sigma} \quad \text{and} \quad (\pi_{y}\mu_{y})(\cup \mathcal{J}_{y}) \geq \delta_{1}^{2\epsilon}. \end{equation}
 This is an analogue of \eqref{form8}, with slightly different notation. By pigeonholing, and restricting to a large subset of $E$, we may assume that $\delta_{1}$ is independent of $y \in E$. Since the upper bound $\delta^{\eta/\sigma} \to 0$ as $\delta \to 0$, we may arrange that $\delta_{1}$ is arbitrarily small.

 Now, following the proof of Proposition \ref{prop2} below \eqref{form8}, one can reach the following conclusions. There exist a scale $\Delta \in [\delta_{1},\delta^{\eta/\sigma}]$, a constant $\zeta = o_{\epsilon}(1)$, and the following objects for some $\Sigma \leq \sigma + \zeta$:
 \begin{enumerate}[(i)]
 \item \label{i'} A non-empty $(\Delta,t,\Delta^{-\zeta})$-set $P_{F} \subset F$,
 \item \label{ii'} A non-empty $(\Delta,s,\Delta^{-\zeta})$-set $P_{E} \subset E$,
 \item \label{iii'} For every $y \in P_{E}$ a non-empty $(\Delta,\Sigma,\Delta^{-\zeta})$-set of $\Delta$-tubes $\mathcal{T}^{y}$ with the properties that $y \in T$ for all $T \in \mathcal{T}^{y}$, and $|T \cap P_{F}| \geq \Delta^{\Sigma + \zeta}|P_{F}|$ for all $T \in \mathcal{T}^{y}$.
 \end{enumerate}
The dependence between "$\epsilon$" and "$\zeta$" is the same as the dependence between the parameters "$\eta$" and "$\epsilon$" in the proof of Proposition \ref{prop2}. Thus, to reach the conclusions \eqref{i'}--\eqref{iii'} for given $\zeta > 0$, we need to choose $\epsilon \ll \zeta^{3/\zeta + 1}$ in \eqref{form54}.

 Once the objects satisfying properties \eqref{i'}--\eqref{iii'} are located, we are in a position to apply Theorem \ref{main2}. The conclusion is that if $\Delta$ is sufficiently small, depending only on $\Sigma,t,\zeta$, then
 \begin{displaymath} s \leq \max\{1 + \Sigma - t,0\} \leq \max\{1 + \sigma + \zeta - t,0\}. \end{displaymath}
 However, we assumed to begin with that $s > \max\{1 + \sigma - t,0\}$. This leads to a contradiction if $\zeta > 0$ is small enough (which can be arranged by picking $\epsilon$ small enough).  \end{proof}


 We will use Proposition \ref{prop4} only via the following corollary:

 \begin{cor}\label{prop1} Let $0 \leq \sigma \leq s \leq 1$, $t \in (1,2]$, $\eta \in (0,1]$, and $s > \max\{1 + \sigma - t,0\}$. Then, there exist
 \begin{displaymath} \delta_{0} = \delta(s,\sigma,t,\eta) > 0 \quad \text{and} \quad \epsilon = \epsilon(s,\sigma,t,\eta) > 0 \end{displaymath}
 such that the following holds for all $\delta \in (0,\delta_{0}]$.

 Let $E,F \subset B(1) \subset \R^{2}$ be non-empty $\delta$-separated sets, where $E$ is a $(\delta,s,\delta^{-\epsilon})$-set, $F$ is a $(\delta,t,\delta^{-\epsilon})$-set, and $\dist(E,F) \geq \tfrac{1}{2}$. Then, there exists a subset $E' \subset E$ with $|E'| \geq (1 - \delta^{\epsilon})|E|$, and for every point $y \in E'$, there exist  disjoint (possibly empty) families of $\delta$-tubes $\mathcal{T}_{1},\ldots,\mathcal{T}_{L}$, where $L = 3\log (1/\delta)$, with the following properties:
 \begin{enumerate}[(a)]
 \item \label{a'} For all $j \in \{1,\ldots,L\}$ and $T \in \mathcal{T}_{j}$, we have $y \in T$.
 \item \label{b'}  Every $\mathcal{T}_{j}$, $j \in \{1,\ldots,L\}$, is a $(\delta,\sigma,\delta^{-\eta})$-set.
 \item \label{c'}  For all $j \in \{1,\ldots,L\}$, we have $2^{j-1} \leq |T \cap F| \leq 2^j$, $T \in \mathcal{T}_{j}$.
 \item \label{d'}  For all $j \in \{1,\ldots,L\}$ we have that either $\mathcal{T}_{j}$ is empty, or $|F \cap (\cup \mathcal{T}_{j})| \geq \delta^{2\epsilon}|F|$.
 \item \label{e'} We have $|F_{\mathrm{bad}}| \leq \delta^{\epsilon}|F|$, where
 \begin{displaymath} F_{\mathrm{bad}} := F \, \setminus \, \bigcup_{j = 1}^{L} \bigcup_{T \in \mathcal{T}_{j}} T. \end{displaymath}
 \end{enumerate}
  \end{cor}


\begin{proof} Instead of finding the subset $E' \subset E$ with $|E'| \geq (1 - \delta^{\epsilon})|E|$, it is sufficient to find a single point $y \in E$ with the stated properties. Indeed, if the claim fails, then the set of "bad" points $y\in E$ has cardinality $\geq \delta^{\epsilon}|E|$ and so is a $(\delta,s,\delta^{-2\epsilon})$-set, and locating a single good point there yields a contradiction.

After this reduction, the point $y \in E$ we are looking for is simply the one given by Proposition \ref{prop4}. Let $\mathcal{T} = \mathcal{T}_{y}$ be an initial collection of $|\mathcal{T}| \sim \delta^{-1}$ tubes of width $\delta$, all of which contain $y$, whose union covers $B(1)$, and whose directions are $\sim \delta$-separated. For $j \in \{1,\ldots,3\log(1/\delta)\}$, we then define
\begin{displaymath} \mathcal{T}_{j}' := \{T \in \mathcal{T} : 2^{j - 1} \leq |F \cap T| < 2^{j}\}. \end{displaymath}
Since $|F| \leq \delta^{-3}$, every point in $F$ is covered by at least one tube in one of the families $\mathcal{T}_{j}'$.

We then plan to refine the families $\mathcal{T}_{j}'$ into $(\delta,\sigma,\delta^{-2\eta})$-sets without discarding too much of the set $F$. Fix $j \in \{1,\ldots,3\log(1/\delta)\}$, and write $F_{j} := F \cap (\cup \mathcal{T}_{j}')$. If $|F_{j}| < 2\delta^{\epsilon}|F|$, we declare that $F_{j} \subset F_{\mathrm{bad}}$, and $\mathcal{T}_{j}=\emptyset$. Otherwise, we construct a $(\delta,s,O(\delta^{-\eta}))$-subset $\mathcal{T}_{j} \subset \mathcal{T}_{j}'$ with the following procedure. Since $|F_{j}| \geq 2\delta^{\epsilon}|F| \geq \delta^{\epsilon}|F|$, we know from Proposition \ref{prop4} (or rather the choice of $y \in E$) that
\begin{equation}\label{form48} \mathcal{H}^{\sigma}_{\delta,\infty}(\pi_{y}(F_{j})) \geq \delta^{\eta}. \end{equation}
Using \cite[Proposition A.1]{FasslerOrponen14}, this allows us to find a Katz-Tao $(\delta,s,C)$-subset $P_{1} \subset \pi_{y}(F_{j})$ of cardinality $|P_{1}| \gtrsim \delta^{\eta - s}$, where $C > 1$ is an absolute constant. Then, we define
\begin{displaymath}
 \mathcal{T}^{1}_{j} := \{T \in \mathcal{T}_{j}' : T \cap \pi_{y}^{-1}(P_{1}) \cap F_{j} \neq \emptyset\}.
 \end{displaymath}
Using $\dist(y,F) \geq \tfrac{1}{2}$, it is easy to check that $\mathcal{T}^{1}_{j}$ is a Katz-Tao $(\delta,s,C')$-set for some $C' \sim C$, and that $|\mathcal{T}_{j}^{1}| \gtrsim \delta^{\eta - s}$.

We then assume that certain disjoint subsets $\mathcal{T}_{j}^{1},\ldots,\mathcal{T}_{j}^{k} \subset \mathcal{T}_{j}'$ have already been constructed, $k \geq 1$. By "disjoint" we mean that the families consist of distinct elements of $\mathcal{T}$; the tubes themselves cannot be disjoint, as they all contain the point $y \in E$. If the cardinality of
\begin{displaymath} F_{j}^{k + 1} := F_{j} \, \setminus \, \cup (\mathcal{T}_{j}^{1} \cup \ldots \cup \mathcal{T}^{k}_{j}) \end{displaymath}
is $|F_{j}^{k + 1}| < \delta^{\epsilon}|F|$, the construction terminates, and the set $F_{j}^{k}$ is added to $F_{\mathrm{bad}}$. Otherwise $|F_{j}^{k + 1}| \geq \delta^{\epsilon}|F|$, and therefore \eqref{form48} holds with $F_{j}^{k + 1}$ in place of $F_{j}$. Repeating the previous steps, this allows us to find another Katz-Tao $(\delta,s,C')$-subset $\mathcal{T}_{j}^{k + 1} \subset \mathcal{T}_{j}'$ with $|\mathcal{T}^{k + 1}_{j}| \gtrsim \delta^{\eta - s}$. We make sure to select only such tubes to $\mathcal{T}_{j}^{k + 1}$ which contain at least one "new" point of $F_{j}$ which is not contained in $\cup (\mathcal{T}^{1}_{j} \cup \ldots \cup \mathcal{T}^{k}_{j})$. Then $\mathcal{T}_{j}^{k + 1} \subset \mathcal{T} \, \setminus \, (\mathcal{T}^{1}_{j} \cup \ldots \cup \mathcal{T}^{k}_{j})$.

The construction terminates in finitely many steps, because $F$ is a finite set: there exists $k = k(j) \geq 1$ such that $|F_{j}^{k + 1}| < \delta^{\epsilon}|F|$. Since initially $|F_{j}| \geq 2\delta^{\epsilon}|F|$, this means that the union $\mathcal{T}_{j}^{1} \cup \ldots \cup \mathcal{T}_{j}^{k}$ covers at least $\delta^{\epsilon}|F|$ points of $F$; this is needed for property \eqref{d'}. At this stage $F_{j}^{k + 1}$ is moved into $F_{\mathrm{bad}}$, and we define $\mathcal{T}_{j} := \mathcal{T}_{j}^{1} \cup \ldots \cup \mathcal{T}_{j}^{k}$. We now check that $\mathcal{T}_{j}$ is a $(\delta,\sigma,O(\delta^{-\eta}))$-set. Indeed, first observe that
\begin{equation}\label{form49} |\mathcal{T}_{j}| = \sum_{i = 1}^{k} |\mathcal{T}_{j}^{i}| \gtrsim k \cdot \delta^{\eta - \sigma}. \end{equation}
On the other hand, each family $\mathcal{T}_{j}^{i}$ was individually a Katz-Tao $(\delta,\sigma,C')$-set. Since all the tubes contain the common point $y$, this is equivalent to saying that the set of directions $S_{j}^{i} \subset S^{1}$ of the tubes in $\mathcal{T}_{j}^{i}$ forms a Katz-Tao $(\delta,\sigma,C'')$-set for some $C'' \sim C'$. 
\begin{displaymath} |S_{j} \cap B(e,r)| \leq \sum_{i = 1}^{k} |S_{j}^{i} \cap B(e,r)| \leq k \cdot C''\left(\frac{r}{\delta} \right)^{\sigma} \stackrel{\eqref{form49}}{\lesssim} \delta^{-\eta} \cdot r^{\sigma}|\mathcal{T}_{j}| \sim \delta^{-\eta} \cdot r^{\sigma}|S_{j}|. \end{displaymath}
Therefore the direction set of $\mathcal{T}_{j}$ is a $(\delta,\sigma,O(\delta^{-\eta}))$-set, and finally the same is true of $\mathcal{T}_{j}$.

To conclude the proof, we recall that $j \in \{1,\ldots,L\}$ with $L = 3\log (1/\delta)$. Therefore the set
\begin{displaymath} F_{\mathrm{bad}} := F_{1}^{k(1) + 1} \cup \ldots \cup F_{L}^{k(L) + 1} \end{displaymath}
satisfies $|F_{\mathrm{bad}}| \lesssim (\log (1/\delta)) \cdot \delta^{\epsilon}|F|$, and in particular $|F_{\mathrm{bad}}| \leq \delta^{\epsilon/2}|F|$ if $\delta = \delta(\epsilon) > 0$ is small enough. We have now established all the conclusions \eqref{a'}--\eqref{e'} with "$\epsilon/2$" in place of "$\epsilon$", and the proof is complete.
\end{proof}

 By standard discretisation arguments (see for example the proof of \cite[Theorem 2]{MR4148151} for a very similar argument), Theorem \ref{main3} follows from its discrete counterpart in the proposition below:
 \begin{proposition}\label{prop3} For every $s \in (0,1]$, $t > 1$, and $\tau \in (0,1]$, there exists $\delta_{0} = \delta_{0}(s,t,\tau) > 0$ and $\epsilon = \epsilon(s,t,\tau) > 0$ such that the following holds for all $\delta \in (0,\delta_{0}]$. Let $E,F \subset B(1) \subset \R^{2}$ be non-empty $\delta$-separated sets, where $E$ is a $(\delta,s,\delta^{-\epsilon})$-set, and $F$ is a $(\delta,t,\delta^{-\epsilon})$-set, and $\dist(E,F) \geq \tfrac{1}{2}$. Then, there exists $x \in F$ such that
 \begin{equation}\label{form37} |\pi_{x}(E')|_{\delta} \geq \delta^{-s + \tau}, \qquad\text{for all } E' \subset E, \, |E'| \geq \delta^{\epsilon}|E|. \end{equation}
 \end{proposition}
 
 In the proof of Proposition \ref{prop3}, it will be useful to assume (additionally) that $|E| \leq \delta^{-s}$. Before giving the details, we reduce matters to that case by a rather abstract argument, which is essentially the same as the proof of \cite[Proposition 25]{MR4148151}.
 
 \begin{lemma}\label{lemma6} Proposition \ref{prop3} follows from its special case where $|E| \leq \delta^{-s}$. More precisely, if the conclusion of Proposition \ref{prop3} holds with this extra assumption and constant "$5\epsilon$", then the conclusion holds without the assumption, with constant "$\epsilon$". \end{lemma}
 
 \begin{proof}  Let $E,F \subset B(1)$ be as in the statement of Proposition \ref{prop3} (but without the assumption $|E| \leq \delta^{-s}$). We claim the existence of $x \in F$ such that \eqref{form37} holds. We make a counter assumption: for every $x \in F$ there exists a set $E_{x} \subset E$ with cardinality $|E_{x}| \geq \delta^{\epsilon}|E|$ with the property $|\pi_{x}(E_{x})|_{\delta} < \delta^{-s + \tau}$. We apply Lemma \ref{lemma3} to the set $E$ with the parameter $\epsilon > 0$. The conclusion is a decomposition
 \begin{equation}\label{form55} E = E_{1} \cup \ldots \cup E_{N} \end{equation}
 into disjoint Katz-Tao $(\delta,s,1)$-sets $E_{j}$ with $N \leq |E|\delta^{s - 2\epsilon}$. Evidently each of the sets $E_{j} \subset B(1)$ satisfies $|E_{j}| \leq \delta^{-s}$, but some of them may be too small to be $(\delta,s,\delta^{-5\epsilon})$-sets, so the special case assumed in Lemma \ref{lemma6} is not immediately applicable to each $E_{j}$. This is easy to fix: we simply discard those sets $E_{j}$ with cardinality $|E_{j}| < \delta^{-s + 4\epsilon}$. The union of such sets, temporarily denoted $B$, has cardinality $|B| < N\delta^{-s + 4\epsilon} \leq \delta^{2\epsilon}|E|$. In particular, $|E_{x} \cap (E \, \setminus \, B)| \geq \delta^{\epsilon}|E| - \delta^{2\epsilon}|E| \gtrsim \delta^{\epsilon}|E|$ for all $x \in F$. This means that $E \, \setminus \, B$ satisfies all the initial assumptions of $E$. Replacing $E$ by $E \, \setminus \, B$ if needed, we may therefore assume that every set $E_{j}$ in the decomposition \eqref{form55} satisfies $|E_{j}| \geq \delta^{-s + 4\epsilon}$. This implies that every $E_{j}$ is a $(\delta,s,\delta^{-4\epsilon})$-set of cardinality $|E_{j}| \leq \delta^{-s}$, and the special case assumed in Lemma \ref{lemma6} is now applicable to every $E_{j}$.
 
We use this as follows. For $j \in \{1,\ldots,N\}$ fixed, we define
\begin{displaymath} F_{j} := \{x \in F : \exists \, E_{x}' \subset E_{j} \text{ s.t. } |E_{x}'| \geq \delta^{5\epsilon}|E_{j}| \text{ and } |\pi_{x}(E_{x}')|_{\delta} < \delta^{s + \tau}\}. \end{displaymath}
We claim that 
\begin{equation}\label{form57} |F_{j}| \leq \delta^{3\epsilon}|F|, \qquad j \in \{1,\ldots,N\}. \end{equation}
Indeed, if $|F_{j}| > \delta^{3\epsilon}|F|$, then $F_{j}$ is a $(\delta,t,\delta^{-4\epsilon})$-set, and contains (by the current hypothesis) a point $x_{0} \in F_{j}$ violating the defining property of $F_{j}$.
 
For $x \in F$ fixed, let $\mathcal{J}_{x} := \{1 \leq j \leq N : |E_{x} \cap E_{j}| \geq \delta^{5\epsilon}|E_{j}|\}$, and note that
 \begin{displaymath} \delta^{\epsilon}|E| \leq |E_{x}| = \sum_{j = 1}^{N} |E_{x} \cap E_{j}| \leq \sum_{j \notin \mathcal{J}_{x}} \delta^{5\epsilon}|E_{j}| + \sum_{j \in \mathcal{J}_{x}} |E_{j}| \leq \delta^{5\epsilon}|E| + \sum_{j \in \mathcal{J}_{x}} |E_{j}|, \end{displaymath}
 and consequently (for $\delta > 0$ small enough)
 \begin{equation}\label{form56} \sum_{j \in \mathcal{J}_{x}} |E_{j}| \geq \delta^{2\epsilon}|E| \quad \Longrightarrow \quad \delta^{-2\epsilon} \sum_{j \in \mathcal{J}_{x}} \frac{|E_{j}|}{|E|} \geq 1. \end{equation} 
 On the other hand, for every $j \in \mathcal{J}_{x}$ we have $|E_{x} \cap E_{j}| \geq \delta^{5\epsilon}|E_{j}|$ and $|\pi_{x}(E_{x} \cap E_{j})|_{\delta} < \delta^{-s + \tau}$. This means (taking $E_{x}' := E_{x} \cap E_{j}$) that $x \in F_{j}$ for all $j \in \mathcal{J}_{x}$. To summarise, for every $x \in F$ there exists a family $\mathcal{J}_{x} \subset \{1,\ldots,N\}$ satisfying \eqref{form56} with the property that $x \in F_{j}$ for all $j \in \mathcal{J}_{x}$. Consequently, 
\begin{displaymath} \mathbf{1}_{F} \leq \delta^{-2\epsilon} \sum_{j = 1}^{N} \frac{|E_{j}|}{|E|} \cdot \mathbf{1}_{F_{j}}, \end{displaymath}
and finally
\begin{align*} |F| \leq \delta^{-2\epsilon} \sum_{j = 1}^{N} \frac{|E_{j}|}{|E|} \cdot |F_{j}| \leq \delta^{-2\epsilon} \max_{1 \leq j \leq N} |F_{j}| \stackrel{\eqref{form57}}{\leq} \delta^{\epsilon}|F|. \end{align*} 
This contradiction completes the proof of the lemma. \end{proof}

We are then equipped to prove Proposition \ref{prop3}

 \begin{proof}[Proof of Proposition \ref{prop3}]  Thanks to Lemma \ref{lemma6}, we may assume that $E$ is a non-empty $(\delta,s,\delta^{-\epsilon})$-set satisfying $|E| \leq \delta^{-s}$. Fix $\eta \in (0,1]$, and let $0 < \epsilon \ll \eta$ be so small that the conclusions of Corollary \ref{prop1} apply with constant "$\eta$" if the hypotheses are in force with constant "$4\epsilon$". We will need to take $\epsilon,\eta > 0$ so small that
\begin{equation}\label{tau} C(\epsilon + \eta)  < \tau \end{equation}
for some absolute constant $C \geq 1$. This determines the necessary size of "$\epsilon$" via Corollary \ref{prop1}. The $t$-dependence of "$\epsilon$" arises from the fact that if $t > 1$, as we assume, then
\begin{displaymath} s > \max\{1 + s - t,0\}. \end{displaymath}
This makes Corollary \ref{prop1} applicable with $\sigma := s$, with $\epsilon = \epsilon(s,t,\tau,\eta)$ (where $\eta = \eta(\tau)$). Now, applying Corollary \ref{prop1} with $\sigma := s$, we may find a subset $E' \subset E$ with $|E'| \geq (1 - \delta^{4\epsilon})|E|$, and for all $y \in E'$ the $(\delta,s,\delta^{-\eta})$-sets $\mathcal{T}^{y}_{1},\ldots,\mathcal{T}_{L}^{y}$ (here $L = 3\log (1/\delta)$) of $\delta$-tubes containing $y$, with $\sim \delta$-separated directions, and the properties \eqref{a'}--\eqref{e'}. In particular, by \eqref{e'},
\begin{equation} \label{Fy}
|F^{y}| \geq (1 - \delta^{4\epsilon})|F|,
\end{equation}
 where
\begin{displaymath} F^{y}_{j} := F \cap (\cup \mathcal{T}_{j}^{y}) \quad \text{and} \quad F^{y} := \bigcup_{j = 1}^{L} F^{y}_{j}. \end{displaymath}
Since $\dist(E,F) \geq \tfrac{1}{2}$, and $y \in E$, the sets $F \cap T$, $T \in \mathcal{T}_{j}^{y}$, have bounded overlap. Using \eqref{c'}-\eqref{d'} in Corollary \ref{prop1}, this yields that either $\mathcal{T}_{j}^y=\emptyset$, or
\begin{equation}\label{form43} \delta^{8\epsilon}|F|/2^{j} \lesssim |F^{y}_{j}|/2^{j} \leq |\mathcal{T}^{y}_{j}| \lesssim |F|/2^{j} \end{equation}
for all $j \in \{1,\ldots,L\}$. 
We define
\begin{displaymath} \mathcal{T}^{y} := \bigcup_{j = 1}^{L} \mathcal{T}^{y}_{j} \quad \text{and} \quad \mathcal{T} := \bigcup_{y \in E'} \mathcal{T}^{y}. \end{displaymath}

Now, we make a counter assumption: \eqref{form37} fails for all $x \in F$. Thus, for every $x \in F$, there exists a subset $E_{x,\mathrm{bad}} \subset E$ such that $|E_{x,\mathrm{bad}}| \geq \delta^{\epsilon}|E|$, and
\begin{equation} \label{Ebad}
|\pi_{x}(E_{x,\mathrm{bad}})|_{\delta} < \delta^{-s+\tau}.
\end{equation}
We note that each $E_{x,\mathrm{bad}}$ has a large intersection with $E'$, which follows from $|E'| \geq (1 - \delta^{4\epsilon})|E|$:
\begin{displaymath} |E_{x,\mathrm{bad}} \cap E'| \geq |E_{x,\mathrm{bad}}| + |E'| - |E| \geq (\delta^{\epsilon} - \delta^{4\epsilon})|E| \gtrsim \delta^{\epsilon}|E|. \end{displaymath}
Given this observation, we may assume that $E_{x,\mathrm{bad}} \subset E'$ in the sequel (by replacing $E_{x,\mathrm{bad}}$ with the intersection $E_{x,\mathrm{bad}} \cap E'$ if necessary). We also define the sets
\begin{displaymath} E_{x,\mathrm{good}}^{j} := \{y \in E' : x \in F^{y}_{j}\} \quad \text{and} \quad E_{x,\mathrm{good}} := \bigcup_{j = 1}^{L} E_{x,\mathrm{good}}^{j} = \{y \in E' : x \in F^{y}\}. \end{displaymath}
We observe that
\begin{displaymath} (1 - \delta^{4\epsilon})|E'||F| \overset{\eqref{Fy}}{\leq} \sum_{y \in E'} |F^{y}| = \sum_{x \in F} |E_{x,\mathrm{good}}|, \end{displaymath}
so there exists a subset $F' \subset F$ with the properties
\begin{displaymath} |F'| \geq (1 - \delta^{2\epsilon})|F| \quad \text{and} \quad |E_{x,\mathrm{good}}| \geq (1 - \delta^{2\epsilon})|E'| \text{ for all } x \in F'. \end{displaymath}
Note that for $x \in F'$, and recalling $E_{x,\mathrm{bad}},E_{x,\mathrm{good}} \subset E'$, we have
\begin{displaymath} |E_{x,\mathrm{bad}} \cap E_{x,\mathrm{good}}| \geq |E_{x,\mathrm{bad}}| + |E_{x,\mathrm{good}}| - |E'| \geq (\delta^{\epsilon} - \delta^{2\epsilon})|E'| \gtrsim \delta^{\epsilon}|E'|. \end{displaymath}
In particular, since $E_{x,\mathrm{good}}$ is a union of the $L = 3 \log(1/\delta)$ sets $E_{x,\mathrm{good}}^{j}$, there exists $j(x) \in \{1,\ldots,L\}$ such that
\begin{displaymath} |E_{x,\mathrm{bad}} \cap E_{x,\mathrm{good}}^{j(x)}| \geq \delta^{2\epsilon}|E'|, \qquad x \in F'. \end{displaymath}

Since $j(x)$ takes $\lesssim \log(1/\delta)$ values, we may fix $j \in \{1,\ldots,L\}$ and $F'' \subset F'$ with $|F''| \geq \delta^{\epsilon}|F'| \gtrsim \delta^{\epsilon}|F|$ such that $j(x) = j$ for all $x \in F''$. We finally write
\begin{displaymath}
E_{x} := E_{x,\mathrm{bad}} \cap E_{x,\mathrm{good}}^{j},\quad x\in F''.
\end{displaymath}

We introduce the "high multiplicity" sub-families
\begin{equation} \label{T-high}
\mathcal{T}_{j}^{\text{high}} := \{T \in \mathcal{T} : |\{y \in E' : T \in \mathcal{T}_{j}^{y}\}| \geq \delta^{-\tau/2}\}, \qquad j \in \{1,\ldots,L\}.
\end{equation}
We claim that for all $x \in F''$, there exists a set of $\delta$-tubes $\mathcal{T}_{x} \subset \mathcal{T}^{\text{high}}_{j}$ containing $x$, and with the properties
\begin{enumerate}[(i)]
\item \label{i''} $|\mathcal{T}_{x}| \lesssim \delta^{-s+\tau}$,
\item \label{ii''} $\sum_{T \in \mathcal{T}_{x}} |\{y \in E' : T \in \mathcal{T}_{j}^{y}\}| \geq \tfrac{1}{2}|E_{x}| \geq \delta^{2\epsilon}|E'|$.
\end{enumerate}
Let, initially, $\mathcal{T}_{x}'$ consist of all tubes in $\mathcal{T}$ which contain $x$, and are contained in $\mathcal{T}_{j}^{y}$ for some $y \in E_{x}$. Since, by Corollary \ref{prop1}\eqref{a'}, such tubes in particular contain a point from $E_{x} \subset E_{x,\mathrm{bad}}$, we have
\begin{displaymath} |\mathcal{T}_{x}'| \lesssim |\pi_{x}(E_{x,\mathrm{bad}})|_{\delta} \overset{\eqref{Ebad}}{\leq} \delta^{-s + \tau}. \end{displaymath}
We will find $\mathcal{T}_{x}$ as a subset of $\mathcal{T}_{x}'$, so \eqref{i''} is automatically satisfied.

Note that if $y \in E_{x}$, then $\mathcal{T}_{x}' \cap \mathcal{T}^{y}_{j} \neq \emptyset$. This is because $y \in E_{x} \subset E_{x,\mathrm{good}}^{j}$ implies $x \in \cup \mathcal{T}^{y}_{j}$ by definition, thus $x\in T$ for some $T\in\mathcal{T}^{y}_{j}$ and then also $T\in\mathcal{T}_{x}'$. As a corollary,
\begin{equation}\label{form42} \sum_{T \in \mathcal{T}_{x}'} |\{y \in E' : T \in \mathcal{T}_{j}^{y}\}| \geq \sum_{y \in E_{x}} |\mathcal{T}_{x}' \cap \mathcal{T}_{j}^{y}| \geq |E_{x}|. \end{equation}
Thus, $\mathcal{T}_{x}'$ satisfies properties \eqref{i''}--\eqref{ii''}, but it may not necessarily be contained in $\mathcal{T}^{\text{high}}_{j}$. It remains to discard the "low multiplicity" tubes $\mathcal{T}_{x}^{\text{low}} := \mathcal{T}_{x}' \, \setminus \, \mathcal{T}^{\text{high}}_{j}$.
By the definition \eqref{T-high},
\begin{displaymath}
\sum_{T \in \mathcal{T}_{x}^{\text{low}}} |\{y \in E' : T \in \mathcal{T}^{y}_{j}\}| \leq |\mathcal{T}_{x}'| \cdot \delta^{-\tau/2} \overset{\eqref{i''}}{\lesssim} \delta^{-s + \tau/2}. \end{displaymath}
Recalling that $|E_{x}| \geq \delta^{2\epsilon}|E'| \geq \delta^{-s + 3\epsilon}$, the right hand side is much smaller than $|E_{x}|$ if the constant "$C$" in \eqref{tau} is large enough. Comparing this with the lower bound in \eqref{form42}, we set $\mathcal{T}_{x} := \mathcal{T}_{x}' \, \setminus \, \mathcal{T}_{x}^{\text{low}} \subset \mathcal{T}^{\text{high}}_{j}$, and we see that property \eqref{ii''} remains valid for $\mathcal{T}_{x}$.

Now, notice that, using $\mathcal{T}_{x} \subset \mathcal{T}^{\text{high}}_{j}$, we have
\begin{align*} \delta^{2\epsilon}|E'||F''| & \stackrel{\eqref{ii''}}{\leq} \sum_{x \in F''} \sum_{T \in \mathcal{T}_{x}} |\{y \in E' : T \in \mathcal{T}^{y}_{j}\}| \\
& = \sum_{y \in E'} \sum_{x \in F''} \sum_{T \in \mathcal{T}_{x}} \mathbf{1}_{\mathcal{T}_{j}^{y} \cap \mathcal{T}^{\text{high}}_{j}}(T)\\
& = \sum_{y \in E'} \sum_{T \in \mathcal{T}_{j}^{y} \cap \mathcal{T}_{j}^{\text{high}}} \sum_{x \in F''} \mathbf{1}_{\mathcal{T}_{x}}(T) \leq \sum_{y \in E'} \sum_{T \in \mathcal{T}_{j}^{y} \cap \mathcal{T}_{j}^{\text{high}}} |T \cap F|. \end{align*}
Recalling that $|F''|\gtrsim \delta^\epsilon|F|$, this implies that there exists a further subset $E'' \subset E'$ with $|E''| \geq \delta^{O(\epsilon)}|E'|$ such that
\begin{displaymath} 2^{j}|\mathcal{T}_{j}^{y} \cap \mathcal{T}^{\text{high}}_{j}| \stackrel{\eqref{c'}}{\geq} \sum_{T \in \mathcal{T}_{j}^{y} \cap \mathcal{T}_{j}^{\text{high}}} |T \cap F| \geq \delta^{O(\epsilon)}|F|, \qquad y \in E''. \end{displaymath}
In particular,
\begin{equation}\label{form50} |\mathcal{T}^{y}_{j} \cap \mathcal{T}_{j}^{\text{high}}| \geq \delta^{O(\epsilon)}|F|/2^{j} \stackrel{\eqref{form43}}{\gtrsim} \delta^{O(\epsilon)}|\mathcal{T}^{y}_{j}|, \qquad y \in E''. \end{equation}
Since $\mathcal{T}^{y}_{j}$ was a $(\delta,s,\delta^{-\eta})$-set for $y \in E'$, this shows that $\mathcal{T}^{y}_{j} \cap \mathcal{T}_{j}^{\text{high}}$ is a $(\delta,s,\delta^{-\eta - O(\epsilon)})$-set of cardinality $\ge \delta^{O(\epsilon)}|F|/2^{j}$, for all $y \in E''$. In the sequel we write
\begin{displaymath}
M := |F|/2^{j}.
\end{displaymath}

Write $\overline{\mathcal{T}} := \mathcal{T}_{j}^{\text{high}} \cap \bigcup_{y \in E''} \mathcal{T}_{j}^{y}$. We define the following set of "incidences":
\begin{displaymath} \mathcal{I}'(\overline{\mathcal{T}},E') := \{(y,T) \in E' \times \overline{\mathcal{T}} : T \in \mathcal{T}_{j}^{y}\}. \end{displaymath}
We seek upper and lower bounds for the cardinality of $\mathcal{I}'(\overline{\mathcal{T}},E')$. For the lower bound, note that $\overline{\mathcal{T}} \subset \mathcal{T}_{j}^{\text{high}}$ and so, recalling the definition \eqref{T-high},
\begin{equation}\label{form44} |\mathcal{I}'(\overline{\mathcal{T}},E')| = \sum_{T \in \overline{\mathcal{T}}} |\{y \in E' : T \in \mathcal{T}_{j}^{y}\}| \geq |\overline{\mathcal{T}}| \cdot \delta^{-\tau/2}. \end{equation}
For the upper bound, we start with Cauchy-Schwarz:
\begin{align} |\mathcal{I}'(\overline{\mathcal{T}},E')| & \leq |\overline{\mathcal{T}}|^{1/2} \Big(\sum_{T \in \overline{\mathcal{T}}} |\{y \in E' : T \in \mathcal{T}_{j}^{y}\}|^{2} \Big)^{1/2} \notag\\
&\label{form45} = |\overline{\mathcal{T}}|^{1/2} \Big( \sum_{y \neq y'\in E'} |\mathcal{T}_{j}^{y} \cap \mathcal{T}_{j}^{y'}| + |\mathcal{I}'(\overline{\mathcal{T}},E')| \Big)^{1/2} \end{align}
Since any tube in $\mathcal{T}_{j}^{y} \cap \mathcal{T}_{j}^{y'}$ passes through $y$ and $y'$, the $(\delta,s,\delta^{-\eta})$-set property of $\mathcal{T}_{j}^{y}$ yields (as usual using directions as a proxy for tubes passing through a fixed point)
\begin{displaymath} |\mathcal{T}_{j}^{y} \cap \mathcal{T}_{j}^{y'}| \lesssim \delta^{-\eta}\left(\frac{\delta}{|y - y'|} \right)^{s}|\mathcal{T}_{j}^{y}| \stackrel{\eqref{form43}}{\lesssim} \delta^{-\eta}\left(\frac{\delta}{|y - y'|} \right)^{s}M. \end{displaymath}
Plugging this back into \eqref{form45}, we find that
\begin{displaymath} |\mathcal{I}'(\overline{\mathcal{T}},E')| \lesssim \big(|\overline{\mathcal{T}}|\delta^{s-\eta}M\big)^{1/2} \Big(\sum_{y \neq y'} \frac{1}{|y - y'|^{s}} \Big)^{1/2} + |\overline{\mathcal{T}}|. \end{displaymath}
The double sum on the right runs over distinct $y,y' \in E' \subset E$, where $E$ was assumed to be a $\delta$-separated $(\delta,s,\delta^{-\epsilon})$-set with $|E| \leq \delta^{-s}$. Therefore, a standard calculation (fix $y \in E$ and split $y' \in E \, \setminus \, \{y\}$ into dyadic annular regions) yields
\begin{displaymath} \sum_{y \neq y'} \frac{1}{|y - y'|^{s}} \lesssim \delta^{-2s - O(\epsilon)}. \end{displaymath}
Therefore, finally,
\begin{displaymath}
|\mathcal{I}'(\overline{\mathcal{T}},E')| \lesssim \delta^{-O(\epsilon+\eta)}|\overline{\mathcal{T}}|^{1/2}(\delta^{-s}M)^{1/2} + |\overline{\mathcal{T}}|.
\end{displaymath}
 When this upper bound is compared with the lower bound in \eqref{form44}, we find that
\begin{equation}\label{form46}
|\overline{\mathcal{T}}| \lesssim \delta^{-O(\epsilon+\eta)} \delta^{-s + \tau}M.
\end{equation}
Recalling that $O(\epsilon+\eta) < \tau$ by \eqref{tau}, we see that \eqref{form46} will now lead to a contradiction if we manage to show that
\begin{equation}\label{form51} |\overline{\mathcal{T}}| \geq \delta^{O(\epsilon+\eta)}\delta^{-s}M. \end{equation}
For this purpose, recall from \eqref{form50} that $\overline{\mathcal{T}}$ contains for every $y \in E''$ the $(\delta,s,\delta^{-\eta - O(\epsilon)})$-set $\mathcal{T}_{j}^{y} \cap \mathcal{T}_{j}^{\text{high}}$ of cardinality $\geq \delta^{O(\epsilon)}M$, and all tubes in $\mathcal{T}_{j}^{y} \cap \mathcal{T}_{j}^{\text{high}}$ contain $y$. Also, $E''$ is a non-empty $(\delta,s,\delta^{-O(\epsilon)})$-set. With these observations in hand, the lower bound \eqref{form51} follows from standard incidence bounds; see e.g. \cite[Corollary 2.14]{2021arXiv210603338O} (the definition of tubes and incidences in that paper is slightly different, but it is not hard to deduce the version we use here from either the statement or the proof of \cite[Corollary 2.14]{2021arXiv210603338O}). This concludes the proof of Proposition \ref{prop3}. \end{proof}

\bibliographystyle{plain}
\bibliography{references}

\end{document}